\documentclass[12pt,reqno]{amsart}

\usepackage{amsmath,amsfonts,amsthm,amssymb,color}
\usepackage[T1]{fontenc}
\usepackage[latin1]{inputenc}

\newcommand{\der}{\delta}

\newcommand{\hz}{\hat z}

\newcommand{\ka}{\kappa}

\newcommand{\lot}{[\ell_1,\ell_2]}
\newcommand{\norm}[1]{\lVert #1\rVert}

\newcommand{\yti}{\tilde{y}}
\newcommand{\zti}{\tilde{z}}

\newcommand{\cyti}{\tilde{\cy}}
\newcommand{\yun}{y^{(1)}}
\newcommand{\zun}{z^{(1)}}
\newcommand{\zde}{z^{(2)}}
\newcommand{\czun}{\cz^{(1)}}
\newcommand{\czde}{\cz^{(2)}}
\newcommand{\hzun}{\hat{z}^{(1)}}
\newcommand{\hzde}{\hat{z}^{(2)}}

\DeclareMathOperator{\id}{\text{Id}}

\newcommand{\cb}{{\mathcal B}}
\newcommand{\cac}{{\mathcal C}}

\newcommand{\cf}{{\mathcal F}}

\newcommand{\cj}{{\mathcal J}}
\newcommand{\cl}{{\mathcal L}}

\newcommand{\cn}{{\mathcal N}}

\newcommand{\cq}{{\mathcal Q}}

\newcommand{\cy}{{\mathcal Y}}
\newcommand{\cZ}{{\mathcal Z}}
\newcommand{\cz}{{\mathcal Z}}

\newcommand{\al}{\alpha}
\newcommand{\be}{\beta}
\newcommand{\ga}{\gamma}
\newcommand{\ep}{\varepsilon}
\newcommand{\la}{\lambda}
\newcommand{\si}{\sigma}
\newcommand{\vp}{\varphi}
\newcommand{\ze}{\zeta}

\newcommand{\laa}{\Lambda}

\newcommand{\N}{{\mathbb N}}

\newcommand{\R}{{\mathbb R}}

\newcommand{\lcl}{\left\{}
\newcommand{\rcl}{\right\}}
\newcommand{\lp}{\left(}
\newcommand{\rp}{\right)}
\newcommand{\lc}{\left[}
\newcommand{\rc}{\right]}
\newcommand{\lln}{\left|}
\newcommand{\rrn}{\right|}

\newcommand{\bean}{\begin{eqnarray*}}
\newcommand{\eean}{\end{eqnarray*}}
\newcommand{\ben}{\begin{enumerate}}
\newcommand{\een}{\end{enumerate}}
\newcommand{\beq}{\begin{equation}}
\newcommand{\eeq}{\end{equation}}

\newtheorem{theorem}{Theorem}[section]

\newtheorem{corollary}[theorem]{Corollary}

\newtheorem{definition}[theorem]{Definition}

\newtheorem{hypothesis}{Hypothesis}
\newtheorem{lemma}[theorem]{Lemma}

\newtheorem{proposition}[theorem]{Proposition}

\theoremstyle{remark}
\newtheorem{remark}[theorem]{Remark}

\begin{document}

\title[Rough Volterra equations]{Rough Volterra equations 1: the algebraic integration setting}
\author{Aurélien Deya and Samy Tindel}

\address{
{\it Aurélien Deya and Samy Tindel:}
{\rm Institut {\'E}lie Cartan Nancy, Nancy-Universit\'e, B.P. 239,
54506 Vand{\oe}uvre-l{\`e}s-Nancy Cedex, France}.
{\it Email: }{\tt deya@iecn.u-nancy.fr}, {\tt tindel@iecn.u-nancy.fr}
}

\begin{abstract}
We define and solve Volterra equations driven by an irregular signal, by means of a variant of the rough path theory called algebraic integration. In the Young case, that is for a driving signal with Hölder exponent $\ga>1/2$, we obtain a global solution, and are able to handle the case of a singular Volterra coefficient. In case of a driving signal with Hölder exponent $1/3<\ga\le 1/2$, we get a local existence and uniqueness theorem. The results are easily applied to the fractional Brownian motion with Hurst coefficient $H>1/3$.
\end{abstract}

\date{\today}



\maketitle


\section{Introduction}

This article is the first of a series of two papers dealing with Volterra equations driven by rough paths. For an arbitrary positive constant $T$, this kind of equation can be written, in its general form, as:
\begin{equation}\label{eq:volterra}
y_t=a+\int_0^t \si(t,u,y_u)\, dx_u,
\quad\mbox{ for }\quad
s\in[0,T],
\end{equation}
where $x$ is a $n$-dimensional Hölder continuous path with Hölder exponent $\ga>0$, $a\in\R^d$ stands for  an initial condition, and $\si:\R_+\times\R_+\times\R^d\to\R^{d,n}$ is a smooth enough function.

\smallskip

Motivated by the previous works on Volterra equations driven by a Brownian motion or a semi-martingale \cite{BM1,BM2,Lew,Pr}, often in an anticipative context \cite{AN,CLP,CD,OZ,NR,PP}, we have taken up the program of defining and solving equation (\ref{eq:volterra}) in a pathwise way, allowing for instance a straightforward application to a fractional Brownian motion with Hurst parameter $H>1/3$. This will be achieved thanks to a variation of the rough path theory due to Gubinelli \cite{Gu}, whose main features are recalled below at Section \ref{sec:alg-integration} (we refer to \cite{FV,Le,LQ-bk} for further classical references on rough paths theory). To the best of our knowledge, this is the first occurrence of a paper dealing with Volterra systems driven by a fractional Brownian motion with $H<1/2$.

\smallskip

More specifically, the current article focuses on the 3 following cases:

\smallskip

\noindent
{\it (i) The Young case:} When 
 $x$ is a $\ga$-Hölder continuous path with $\ga>1/2$ (in particular for a $n$-dimensional fBm with Hurst parameter $H\in (1/2,1)$), and assuming that $\si: [0,T]^2 \times \R^d \to \R^{d,n}$ is regular enough (with respect to its three variables), we shall prove that equation (\ref{eq:volterra}) can be interpreted and solved in the Young sense (Section \ref{sec:young-case}).

\smallskip

\noindent
{\it (ii) The Young singular case:}
Under the same conditions as in the previous case for $x$, we are able to handle the case of a coefficient $\si$ admitting a singularity with respect to its first two variables $t,u$. Namely, if $\si$ can be expressed as $\si(t,u,z)=(t-u)^{-\al} \psi(z)$, for some $\al >0$ and $\psi: \R^d \to \R^{d,n}$ regular enough, then under some conditions on $\al,\ga,\ka$ (roughly speaking, we ask that $\ga-\al>1/2$ and $1/2<\ka<\ga$), it is still possible to interpret $\int_0^t \si(t,u,y_u) \, dx_u$ as a Young integral when $y$ belongs to a space of $\ka$-Hölder functions, denoted below by $\cac_1^\ka([0,T],\R^d)$. This extension of the Young integral requires however a careful analysis, which will be detailed at  Section \ref{sec:singular-case}. We can then solve equation~(\ref{eq:volterra}) in the space $\cac_1^\ka([0,T],\R^d)$.

\smallskip

\noindent
{\it (iii) The rough case:}
When $x$ is a $\ga$-Hölder signal with $\ga \in (1/3,1/2)$ (this applies obviously to a $n$-dimensional fBm with Hurst parameter $H \in (1/3,1/2)$), the integral appearing in equation~(\ref{eq:volterra}) has then to be interpreted in some rough path sense. As mentioned before, we shall resort in this case to the formalism introduced in \cite{Gu}, which allows us to prove the existence and uniqueness of a local solution, defined on a small interval $[0,T_0]$ for some $T_0 \in (0,T]$ (Section \ref{sec:rough-case}). We will then point out the technical difficulties one must cope with when trying to extend this local solution.

\smallskip

Here is a brief sketch of the strategy we have followed in order to obtain our results: the algebraic integration formalism relies heavily on the notion of increments, which are simply given, in case of a function $y$ of one parameter $t\in[0,T]$, by $(\der y)_{st}=y_t-y_s$. At a heuristic level, the main difference between classical differential equations driven by rough signals and our Volterra setting lies in the dependance of the increment $(\der y)_{st}$ of the possible solution on the whole past of the trajectory. Indeed, if $y$ is a solution to equation (\ref{eq:volterra}), then one has
\begin{equation}\label{decomposition-generale}
(\der y)_{st}=\int_s^t \si(t,u,y_u) \, dx_u+\int_0^s \lc \si(t,u,y_u)-\si(s,u,y_u)\rc \, dx_u.
\end{equation}
As one might expect, the first integral in (\ref{decomposition-generale}) can be dealt with just as the classical diffusion case treated in \cite{Gu}. In other words, under suitable regularity conditions on $\si$, the variable $t$ appearing in the integrand does not play a prominent role. The second term in the right hand side of (\ref{decomposition-generale}) is the one which is typical of the Volterra setting, and involves the whole past of $x$. It is still possible to retrieve some $\lln t-s \rrn$-increments from this term thanks to the regularity of $\si$ with respect to its first variable, in order to solve our equation by a fixed point argument. However, as we shall see at Section \ref{sec:5.3}, the term $\int_0^s \lc \si(t,u,y_u)-\si(s,u,y_u)\rc \, dx_u$ will eventually induce some severe problems in the classical arguments allowing to get a global solution for our differential system in the rough case. This explains why we have decided to change radically the setting presented here in the companion paper \cite{DT}. In this latter reference, by means of what we call generalized convolutional increments, we show how to get a global solution to equation (\ref{eq:volterra}) in case of a rough driving noise $x$, for a wide class of coefficients $\si$. It was however important for us to include also a direct treatment of Volterra systems by existing rough paths methods, mainly because (i) It allows to consider a more general driving coefficient $\si$. (ii) The method presented here works perfectly well for the Young setting, and can be further extended in order to cover the case of a singular coefficient $\si$.

\smallskip

Here is how our paper is structured: we recall at Section \ref{sec:alg-integration} the notions of algebraic integration which will be needed later on. Section \ref{sec:young-case} is devoted to the study of equation~(\ref{eq:volterra}) driven by a $\ga$-Hölder continuous process with $\ga>1/2$, when the coefficient $\si$ is regular. Section \ref{sec:singular-case} deals with the same kind of equation, with a singular coefficient $\si$. Section \ref{sec:rough-case} treats the case of a rough driving signal $x$, and finally the proof of some technical lemmas are postponed to the Appendix.

\smallskip

Let us finish this introduction by fixing some notations which are used throughout the paper: we call $Df$ the gradient of a function $f$, defined on $\R^n$, and when we want to stress the fact that we are differentiating $f$ with respect to the $j\textsuperscript{th}$ variable, we denote this by $D_jf$.
As far as the regularity of $\si$ is concerned, the following spaces come into play. If $E,F$ are Banach spaces and $U$ an open set of $E$, denote $\cac^{n,\textbf{\textit{b}}}(U;F)$ the set of $n$-times differentiable mappings from $U$ to $F$ with bounded derivatives. For each $\ka \in (0,1)$, let us also introduce the subset
$$\cac^{n,\textbf{\textit{b}},\ka}(U;F)=\lcl \si \in \cac^{n,\textbf{\textit{b}}}(U;F): \ \sup_{x,y \in U} \frac{\norm{D^{(n)}\si(x)-D^{(n)}\si(y)}}{\norm{x-y}^\ka} < \infty \rcl.$$


\section{Algebraic integration}
\label{sec:alg-integration}

The current section is devoted to recall the main concepts of algebraic integration, which will be essential in order to define suitable notions of generalized integrals in our setting. Namely, we shall recall the definition of the spaces of increments $\cac_n^{\ka}$, of the operator $\delta$, and its inverse called $\Lambda$ (or sewing map according to the terminology of \cite{FP}). We will also recall some elementary but useful algebraic relations on the spaces of increments.

\subsection{Increments}\label{sec:incr}

As mentioned in the introduction, the extended integral we deal
with is based on the notion of increment, together with an
elementary operator $\der$ acting on them. The notion of increment can be introduced in the following way:  for two arbitrary real numbers
$\ell_2>\ell_1\ge 0$, a vector space $V$, and an integer $k\ge 1$, we denote by
$\cac_k(V)$ the set of continuous functions $g : [\ell_1,\ell_2]^{k} \to V$ such
that $g_{t_1 \cdots t_{k}} = 0$
whenever $t_i = t_{i+1}$ for some $i\le k-1$.
Such a function will be called a
\emph{$(k-1)$-increment}, and we will
set $\cac_*(V)=\cup_{k\ge 1}\cac_k(V)$. The operator $\der$
alluded to above can be seen as an operator acting on
$k$-increments, 
and is defined as follows on $\cac_k(V)$:
\begin{equation}
  \label{eq:coboundary}
\delta : \cac_k(V) \to \cac_{k+1}(V) \qquad
(\delta g)_{t_1 \cdots t_{k+1}} = \sum_{i=1}^{k+1} (-1)^i g_{t_1
  \cdots \hat t_i \cdots t_{k+1}} ,
\end{equation}
where $\hat t_i$ means that this particular argument is omitted.
Then a fundamental property of $\der$, which is easily verified,
is that
$\delta \delta = 0$, where $\delta \delta$ is considered as an operator
from $\cac_k(V)$ to $\cac_{k+2}(V)$.
 We will denote $\cZ\cac_k(V) = \cac_k(V) \cap \text{Ker}\delta$
and $\cb \cac_k(V) =\cac_k(V) \cap \text{Im}\delta$.

\smallskip

Some simple examples of actions of $\der$,
which will be the ones we will really use throughout the paper,
 are obtained by letting
$g\in\cac_1$ and $h\in\cac_2$. Then, for any $t,u,s\in\lot$, we have
\begin{equation}
\label{eq:simple_application}
  (\der g)_{st} = g_t - g_s,
\quad\mbox{ and }\quad
(\der h)_{sut} = h_{st}-h_{su}-h_{ut}.
\end{equation}
Furthermore, it is readily checked that
the complex $(\cac_*,\delta)$ is \emph{acyclic}, i.e.
$\cZ \cac_{k}(V) = \cb \cac_{k}(V)$ for any $k\ge 1$. In particular, the following basic property, which we
label  for further use, holds true:
\begin{lemma}\label{exd}
Let $k\ge 1$ and $h\in \cz\cac_{k+1}(V)$. Then there exists a (non unique)
$f\in\cac_{k}(V)$ such that $h=\der f$.
\end{lemma}

\noindent
Observe that Lemma \ref{exd} implies that all the elements
$h \in\cac_2(V)$ such that $\der h= 0$ can be written as $h = \der f$
for some (non unique) $f \in \cac_1(V)$. Thus we get a heuristic
interpretation of $\der |_{\cac_2(V)}$:  it measures how much a
given 1-increment  is far from being an {\it exact} increment of a
function (i.e. a finite difference).

\smallskip

Notice that our future discussions will mainly rely on
$k$-increments with $k \le 2$, for which we will use some
analytical assumptions. Namely,
we measure the size of these increments by H\"older norms
defined in the following way: for $f \in \cac_2(V)$ let
$$
\norm{f}_{\mu} \equiv
\sup_{s,t\in\lot}\frac{|f_{st}|}{|t-s|^\mu},
\quad\mbox{and}\quad
\cac_1^\mu(V)=\lcl f \in \cac_2(V);\, \norm{f}_{\mu}<\infty  \rcl.
$$
In the same way, for $h \in \cac_3(V)$, set
\begin{eqnarray}
  \label{eq:normOCC2}
  \norm{h}_{\gamma,\rho} &=& \sup_{s,u,t\in\lot}
\frac{|h_{sut}|}{|u-s|^\gamma |t-u|^\rho}\\
\norm{h}_\mu &\equiv &
\inf\left \{\sum_i \norm{h_i}_{\rho_i,\mu-\rho_i} ;\, h  =\sum_i h_i,\, 0 < \rho_i < \mu \right\} ,\nonumber
\end{eqnarray}
where the last infimum is taken over all sequences $\{h_i \in \cac_3(V) \}$ such that $h
= \sum_i h_i$ and for all choices of the numbers $\rho_i \in (0,z)$.
Then  $\norm{\cdot}_\mu$ is easily seen to be a norm on $\cac_3(V)$, and we set
$$
\cac_3^\mu(V):=\lcl h\in\cac_3(V);\, \norm{h}_\mu<\infty \rcl.
$$
Eventually,
let $\cac_3^{1+}(V) = \cup_{\mu > 1} \cac_3^\mu(V)$,
and remark that the same kind of norms can be considered on the
spaces $\cZ \cac_3(V)$, leading to the definition of some spaces
$\cZ \cac_3^\mu(V)$ and $\cZ \cac_3^{1+}(V)$. In order to avoid ambiguities, we shall denote by $\cn[f;\, \cac_j^\kappa]$ the $\kappa$-Hölder norm on the space $\cac_j$, for $j=1,2,3$. For $\zeta\in\cac_j(V)$, we also set $\mathcal{N}[\zeta;\mathcal{C}_{j}^{0}(V)]=\sup_{s\in[\ell_1; \ell_2]^j}\lVert \zeta_s\rVert_{V}$.

\vspace{0.2cm}

Recall that Lemma \ref{exd} states that for any $h\in\cz\cac_3$, there exists a $f\in\cac_2$ such that $\der f=h$. Importantly enough for the construction of our generalized integrals, this increment $f$ is unique under some additional regularity conditions expressed in terms of the Hölder spaces we have just introduced:
\begin{theorem}[The sewing map] \label{prop:Lambda}
Let $\mu >1$. For any $h\in \cz \cac_3^\mu([0,1]; V)$, there exists a unique $\Lambda h \in \cac_2^\mu([0,1];V)$ such that $\der( \Lambda h )=h$. Furthermore,
\begin{eqnarray} \label{contraction}
\norm{ \Lambda h}_\mu \leq c_\mu  \, \cn[h;\, \cac_3^{\mu}(V)],
\end{eqnarray}
with $c_\mu =2+2^\mu \sum_{k=1}^\infty k^{-\mu}$. This gives rise to a linear continuous map $\laa:  \cz \cac_3^\mu([0,1]; V) \rightarrow \cac_2^\mu([0,1];V)$ such that $\der \laa =\id_{ \cz \cac_3^\mu([0,1]; V)}$.
\end{theorem}

\begin{proof}
The original proof of this result can be found in \cite{Gu}. We refer to \cite{DT,GT} for two simplified versions.

\end{proof}

At this point the connection of the structure we introduced with
the problem of integration of irregular functions can be still quite
obscure to the non-initiated reader. However something interesting is
already going on and the previous corollary has a very nice
consequence which is the subject of the following property.

\begin{corollary}[Integration of small increments]
\label{cor:integration}
For any 1-increment $g\in\cac_2 (V)$, such that $\der g\in\cac_3^{1+}$,
set
$
\delta f = (\id-\Lambda \delta) g
$.
Then
$$
(\delta f)_{st} = \lim_{|\Pi_{st}| \to 0} \sum_{i=0}^n g_{t_{i} t_{i+1}},
$$
where the limit is over any partition $\Pi_{st} = \{t_0=s,\dots,
t_n=t\}$ of $[s,t]$ whose mesh tends to zero. The
1-increment $\delta f$ is the indefinite integral of the 1-increment $g$.
\end{corollary}
\begin{proof}
For any partition $\Pi_{t}=\{s=t_0 < t_1 <...<t_n=t\}$ of $[s,t]$, write
$$(\der f)_{st}=\sum_{i=0}^n (\der f)_{t_{i}t_{i+1}} =\sum_{i=0}^n g_{t_{i}t_{i+1}}-\sum_{i=0}^n \laa_{t_{i}t_{i+1}}(\der g).$$
Observe now that for some $\mu >1$ such that $\der g \in \cac_3^\mu$,
$$\norm{\sum_{i=0}^n \laa_{t_{i}t_{i+1}}(\der g)}_V \leq \sum_{i=0}^n \norm{\laa_{t_{i}t_{i+1}}(\der g)}_V \leq \norm{\laa(\der g)}_\mu \, \lln \Pi_{st}\rrn^{\mu-1} \, \lln t-s\rrn,$$
and as a consequence, $\lim_{\lln \Pi_{st}\rrn \rightarrow 0} \sum_{i=0}^n \laa_{t_{i}t_{i+1}}(\der g) =0$.
\end{proof}

\subsection{Computations in $\cac_*$}\label{cpss}

We gather in this section some elementary but useful algebraic rules for increments. We refer again to \cite{DT,GT} for the proof of these statements.

\smallskip

For sake of simplicity, let us assume for the moment
that $V=\R$ (the multidimensional version of the below considerations can be found in \cite{NNRT}), and set $\cac_k(\R)=\cac_k$. Then
the complex $(\cac_*,\delta)$ is an (associative, non-commutative)
graded algebra once endowed with the following product:
for  $g\in\cac_n $ and $h\in\cac_m $ let  $gh \in \cac_{n+m} $
the element defined by
\begin{equation}\label{cvpdt}
(gh)_{t_1,\dots,t_{m+n-1}}g_{t_1,\dots,t_{n}} h_{t_{n},\dots,t_{m+n-1}},
\quad
t_1,\dots,t_{m+n+1}\in\lot.
\end{equation}
In this context, we have the following useful properties.

\begin{proposition}\label{difrul}
The following differentiation rules hold true:
\begin{enumerate}
\item
Let $g,h$ be two elements of $\cac_1 $. Then
\begin{equation}\label{difrulu}
\der (gh) = \der g\,  h + g\, \der h.
\end{equation}
\item
Let $g \in \cac_1 $ and  $h\in \cac_2 $. Then
$$
\der (gh) = \der g\, h + g \,\der h, \qquad
\der (hg) = \der h\, g  - h \,\der g.
$$
\end{enumerate}
\end{proposition}

\vspace{0.2cm}

The iterated integrals of smooth functions on $\lot$ are obviously
particular cases of elements of $\cac$ which will be of interest for
us, and let us recall  some basic  rules for these objects:
consider $f,g\in\cac_1^\infty $, where $\cac_1^\infty $ is the set of
smooth functions from $\lot$ to $\R$. Then the integral $\int dg \,
f$, which will be denoted by
$\cj(dg \,  f)$, can be considered as an element of
$\cac_2^\infty$. That is, for $s,t\in\lot$, we set
$$
\cj_{st}(dg \,  f)\left(\int  dg f \right)_{st} = \int_s^t  dg_u f_u.
$$
The multiple integrals can also be defined in the following way:
given a smooth element $h \in \cac_2^\infty$ and $s,t\in\lot$, we set
$$
\cj_{st}(dg\, h )\equiv
\left(\int dg h \right)_{st} = \int_s^t dg_u h_{us} .
$$
In particular, the double integral $\cj_{st}( df^3df^2\,f^1)$ is defined, for
$f^1,f^2,f^3\in\cac_1^\infty$, as
$$
\cj_{st}( df^3df^2\,f^1)
=\lp \int df^3df^2\,f^1  \rp_{st}
= \int_s^t df_u^3 \,\cj_{us}\lp  df^2 \, f^1 \rp .
$$
Now, suppose that the $n\textsuperscript{th}$ order iterated integral of $df^n\cdots df^2 \,f^1$, still denoted by $\cj(df^n$ $\cdots df^2 \,f^1)$, has been defined for
$f^1,f^2\ldots, f^n\in\cac_1^\infty$.
Then, if $f^{n+1}\in\cac_0^\infty$, we set
\begin{equation}\label{multintg}
\cj_{st}(df^{n+1}df^n \cdots df^2 f^1)\int_s^t  df_u^{n+1}\, \cj_{us}\lp df^n\cdots df^2 \,f^1\rp\,,
\end{equation}
which defines the iterated integrals of smooth functions recursively.
Observe that a $n$th order integral $\cj(df^n\cdots df^2 df^1)$ (instead of $\cj(df^n\cdots df^2 f^1)$) could be defined along the same lines.

\medskip

The following relations between multiple integrals and the operator $\der$ will also be useful in the remainder of the paper:
\begin{proposition}\label{dissec}
Let $f,g$ be two elements of $\cac_1^\infty$. Then, recalling the convention
(\ref{cvpdt}), it holds that
$$
\der f = \cj( df), \qquad
\der\lp \cj( dg f)\rp = 0, \qquad
\der\lp \cj (dg df)\rp =  (\der g) (\der f) = \cj(dg) \cj(df),
$$
and, in general,
$$
 \der \lp \cj( df^n \cdots df^1)\rp   =  \sum_{i=1}^{n-1}
\cj\lp df^n \cdots df^{i+1}\rp \cj\lp df^{i}\cdots df^1\rp.
$$
\end{proposition}


\section{The Young case}
\label{sec:young-case}
In this section, we assume that the driving process $x$ of equation (\ref{eq:volterra}) is a continuous process in $\cac_1^\ga([0,T]; \R^{n})$, for some $\ga \in (1/2,1)$. If $z\in \cac_1^\rho([0,T];\R^{d,n})$, the formalism introduced in the previous section enables to give a meaning to the integral $\int_s^t z_u \, dx_u$ when $\rho+\ga >1$, in the Young sense. This is the issue of the following proposition, borrowed from \cite[Proposition 3]{Gu}:

\begin{proposition}\label{prop:3.1}
If $z\in \cac_1^\rho([0,T];\R^{d,n})$ for some $\rho >0$ such that $\rho+\ga >1$, we can define, for any $s,t \in [0,T]$, 
\begin{equation} \label{integrale-young-exacte}
\cj_{st}(z \, dx):=z_s (\der x)_{st}-\laa_{st}(\der z \, \der x).
\end{equation}
Then $\cj(z \, dx) \in \cac_2^\ga([0,T];\R^d)$ and 
\begin{equation}\label{inegalite-integrale-volterra}
\cn[\cj(z \, dx); \cac_2^\ga([0,T];\R^d)] \leq c_x \lcl \cn[z; \cac_1^0([0,T];\R^{d,n})]+T^\rho \cn[z; \cac_1^\rho([0,T];\R^{d,n})] \rcl.
\end{equation}
\end{proposition}

\begin{remark}
Thanks to Corollary \ref{cor:integration}, $\cj_{st}(z \, dx)$ can also be seen as a Young integral, that is 
\begin{equation}\label{integrale:version-limite}
\cj_{st}(z \, dx)=\lim_{|\Delta| \to 0} \sum_{\Delta} z_{t_i} (\der x)_{t_i t_{i+1}}.
\end{equation}
Nevertheless, as we shall see in a moment, the exact expression (\ref{integrale-young-exacte}) of the integral is easier to deal with for computational purposes than the limit expression (\ref{integrale:version-limite}), owing to a better knowledge of the remainder $\Lambda(\der z \, \der x)$.
\end{remark}

With this definition in mind, the Volterra equation (\ref{eq:volterra}) will now be interpreted in the Young sense, and is written as:
\begin{equation}\label{eq:volterra-young}
y_t=a+\cj_{0t}(\si(t,.,y_.) \, dx ).
\end{equation}
The next lemma ensures that the latter integral is well-defined:
\begin{lemma}\label{verif-integrabilite}
If $y\in \cac_1^\ga([0,T];\R^d)$ and $\si \in \cac^{1,\textbf{\textit{b}}}([0,T]^2 \times \R^d; \R^{d,n})$, then, for any $t\geq 0$, $\si(t,.,y_.) \in \cac_1^\ga([0,T]; \R^{d,n})$ and
\begin{equation}\label{inegalite-sigma-0}
\cn[\si(t,.,y_.);\cac_1^\ga] \leq c_\si  (T^{1-\ga}+\cn[y;\cac_1^\ga]).
\end{equation}
\end{lemma}
\begin{proof}
This is obvious: recall that we denote by $D\si$ the gradient of $\si$. Then, if $0\leq u <v \leq T$ we get:
$$\norm{\si(t,v,y_v)-\si(t,u,y_u)} \leq \norm{D\si}_\infty \lp |v-u|+\cn[y;\cac_1^\ga] |v-u|^\ga \rp.$$
Hence $\cn[\si(t,.,y_.);\cac_1^\ga] \leq \norm{D\si}_\infty (T^{1-\ga}+\cn[y;\cac_1^\ga])$.

\end{proof}

We are now in position to prove the announced existence and uniqueness result for the Volterra equation in the Young case:
\begin{theorem}\label{theo:young-case}
Assume that the driving process $x$ is an element of $\cac_1^\ga([0,T];\R^n)$ with $\ga>1/2$. Let $\ka \in (0,1)$ such that $\ka(1+\ga) >1$, $a\in \R^d$, $\si \in \cac^{2,\textbf{\textit{b}},\ka}([0,T]^2 \times \R^d; \R^{d,n})$. Then Equation (\ref{eq:volterra-young}) admits a unique solution in $\cac_1^\ga([0,T];\R^d)$.
\end{theorem}
This theorem can be obviously applied to the fractional Brownian motion, in the following sense:
\begin{corollary}
Let $B$ be a $n$-dimensional fractional Brownian motion with Hurst parameter $H>1/2$, defined on a complete probability space $(\Omega,\cf,P)$. Then almost surely, $B$ fulfills the hypotheses of Theorem \ref{theo:young-case}.
\end{corollary}

We divide the proof of Theorem \ref{theo:young-case} into two propositions: first, we will look for a local solution defined on some interval $[0,T_0]$ with $0<T_0\leq T$, and then we will settle a patching argument to extend it onto the whole interval $[0,T]$.

\

\noindent\textbf{Notations}. Before going into the details of the proof, let us mention a few conventions that will be used in the sequel. We assume that we always work with a fixed (finite) horizon $T$ to be distinguished from the intermediate times $T_1,T_0,...$. In particular, this means that the constants that will appear in the below calculations may depend on $T$ without explicit note.\\
For the sake of conciseness, let us denote $\cy_u=(u,y_u) \in [0,T] \times \R^d$ and $\si^t(\cy_u)=\si(t,\cy_u)$.

\

The local existence and uniqueness result for our Volterra equation is contained in the following:
\begin{proposition}\label{prop:3.6}
Under the hypothesis of Theorem \ref{theo:young-case}, there exists $T_0 \in (0,T]$ such that Equation (\ref{eq:volterra-young}) admits a unique solution in $\cac_1^\ga([0,T_0];\R^d)$.
\end{proposition}

\begin{proof}
We are going to resort to a fixed point argument. To this end, let us associate to each $y\in \cac_1^\ga([0,T_0])$ the element $z=\Gamma(y)$ defined by
$$z_t=\Gamma(y)_t=y_0+\cj_{0t}(\si^t(\cy_.) \, dx).$$
The solution we are looking for will then be constructed as a fixed point of $\Gamma$. 

\smallskip

\noindent
{\it Step 1: Invariance of a ball}. Fix a time $T_1 \in (0,T]$ ($T_1$ will be chosen retrospectively). Let $y\in \cac_1^\ga([0,T_1])$ such that $y_0=a$ and set $z=\Gamma(y)$, where, of course, the application $\Gamma$ has been adapted to $[0,T_1]$. 

\smallskip

At this point, let us remind the reader of some specificity of the Volterra setting that we evoked in the introduction. As in (\ref{decomposition-generale}), the increment $(\der z)_{ts}$ can be decomposed as a sum of two terms that will receive a distinct treatment: $I^1_{st}=\cj_{st}(\si^t(\cy) \, dx)$ and $I^2_{st}=\cj_{os}([\si^t-\si^s](\cy) \, dx)$. In order to estimate those two integrals, we shall of course resort to inequality (\ref{inegalite-integrale-volterra}). However, as far as $I^2_{st}$ is concerned, it is clear that the latter inequality will not be sufficient so as to retrieve $\lln t-s \rrn$-increments (remember that we are looking for an estimation of $\cn[z;\cac_1^\ga]$, hence a relation of the form $\norm{I^2_{st}} \leq \lln t-s \rrn^\ga f(y)$). This is where the following lemma, which also anticipates the contraction argument, will come into play.

\begin{lemma}\label{estimation-sigma}
Let $I=[a,b] \subset [0,T]$ and $y, \yti \in \cac_1^{\ga}(I;\R^d)$ such that $y_a=\yti_a$. Then, under the hypothesis of Theorem \ref{theo:young-case}, for any $s,t \in I$,
\begin{equation}\label{inegalite-sigma-1}
\cn[[\si^t-\si^s](\cy); \cac_1^\ga(I) ] \leq c_\si  \lln t-s \rrn \lcl 1+\cn[y;\cac_1^\ga(I)] \rcl,
\end{equation}

\begin{equation}\label{inegalite-sigma-2}
\cn[\si^t(\cy)-\si^t(\cyti); \cac_1^\ga(I) ] \leq c_\si  \lcl 1+\cn[y; \cac_1^\ga(I)]+\cn[\yti;\cac_1^\ga(I)] \rcl \, \cn[y-\yti; \cac_1^\ga(I) ],
\end{equation}

\begin{multline}\label{inegalite-sigma-3}
\cn[[\si^t-\si^s](\cy)-[\si^t-\si^s](\cyti); \cac_1^{\ka\ga}(I) ]\\ \leq c_\si  \lln t-s \rrn \lcl 1+\cn[y; \cac_1^\ga(I)]^\ka+\cn[\yti;\cac_1^\ga(I)]^\ka \rcl \, \cn[y-\yti; \cac_1^\ga(I) ].
\end{multline}
\end{lemma}
\begin{proof}
See Appendix.
\end{proof}

Now, let us go into the details. To deal with $I^1$, use (\ref{inegalite-integrale-volterra}) to get
\bean
\norm{I^1_{st} } &\leq & c_x \lln t-s \rrn^\ga \lcl \cn[\si^t(\cy); \cac_1^0] +T_1^\ga \cn[\si^t(\cy);\cac_1^\ga] \rcl \\
&\leq & c_{x,\si} \lln t-s\rrn^\ga \lcl 1+T_1^\ga \cn[\si^t(\cy); \cac_1^\ga] \rcl,
\eean
and thus, thanks to Lemma \ref{verif-integrabilite}, $\cn[I^1; \cac_2^\ga] \leq c_{x,\si} \lcl 1+T_1^\ga \cn[y;\cac_1^\ga]\rcl$.

\smallskip

Split $I^2$ into $I^2=I^{2,1}+I^{2,2}$, with
$$I^{2,1}_{st}= [ \si^t-\si^s](\cy_0) \, (\der x)_{0s} \quad \mbox{and} \quad I^{2,2}_{st}=\laa_{0s}( \der([ \si^t-\si^s](\cy)) \, \der x).$$
First, notice that $\norm{I^{2,1}_{st} } \leq \norm{D\si }_\infty \lln t-s\rrn \cn[x;\cac_1^\ga] T_1^\ga$, which gives $\cn[I^{2,1}; \cac_2^\ga] \leq c_{x,\si} T_1$. As for $I^{2,2}$, use the contraction property (\ref{contraction}) and the estimate (\ref{inegalite-sigma-1}) to deduce
\bean
\norm{I_{st}^{2,2} }& \leq & c_x \, \cn[ [\si^t-\si^s](\cy); \cac_1^\ga ]\, T_1^\ga\\
&\leq & c_{x,\si} \, \lln t-s \rrn \lcl 1+\cn[y; \cac_1^\ga] \rcl T_1^{2\ga},  
\eean
so that $\cn[I^{2,2}; \cac_2^\ga]  \leq c_{x,\si} \, T_1^{1+\ga} \lp 1+\cn[y;\cac_1^\ga] \rp$.

\smallskip

Therefore, putting together our bounds on $I^1$ and $I^2$, we have obtained $\cn[z;\cac_1^\ga ] \leq c_{x,\si} \lcl 1+T_1^\ga \cn[y;\cac_1^\ga] \rcl$. We can thus pick $T_1 \in (0,T]$ such that for each $0 < T_0 \leq T_1$, there exists a radius $A_{T_0}$ for which the ball
$$B_{T_0,a}^{A_{T_0}} =\lcl y\in \cac_1^\ga([0,T_0]): \ y_0=a, \ \cn[y;\cac_1^\ga([0,T_0])] \leq A_{T_0} \rcl$$
is invariant by $\Gamma$. Notice that the radius $A_{T_0}$ is an increasing function of $T_0$, a fact which will be used in the second step.

\smallskip

\noindent
{\it Step 2: Contraction property}. Fix a time $T_0 \in (0,T_1]$ and let $y,\yti \in B_{T_0,a}^{A_{T_0}}$. Set $z=\Gamma(y), \zti=\Gamma(\yti)$ and decompose again $\der(z-\zti)$ into $\der(z-\zti)=J^{1,1}+J^{1,2}+J^2$, with
$$J^{1,1}_{st}=(\si^t(\cy_s)-\si^t(\cyti_s)) \, (\der x)_{st} \quad , \quad J_{st}^{1,2}=\laa_{st} \lp \der(\si^t(\cy)-\si^t(\cyti)) \, \der x \rp,$$
$$J_{st}^2=\laa_{0s} \lp \der( [\si^t-\si^s](\cy)-[\si^t-\si^s](\cyti)) \, \der x \rp.$$
Let us now estimate the $\ga$-Hölder norm of each of these three terms.

\smallskip

\noindent
{\it Case of} $J^{1,1}$: We have $\cn[J^{1,1}; \cac_2^\ga ] \leq \norm{D\si}_\infty \, \cn[y-\yti; \cac_1^0] \, \cn[x;\cac_1^\ga]$. However, since $y_0=\yti_0=a$, we have $y_s-\yti_s= y_s-\yti_s-(y_0-\yti_0)$, $\cn[y-\yti;\cac_1^0] \leq \cn[y-\yti; \cac_1^\ga ] \, T_0^\ga$, so that
$$\cn[J^{1,1}; \cac_2^\ga ] \leq c_{x,\si} \, \cn[y-\yti; \cac_1^\ga ] \, T_0^\ga.$$

\smallskip

\noindent
{\it Case of} $J^{1,2}$: Inequalities (\ref{contraction}) and (\ref{inegalite-sigma-2}) yield:
\bean
\norm{J_{st}^{1,2} } &\leq & c \, \cn[\si^t(\cy)-\si^t(\cyti); \cac_1^\ga ] \, \cn[x;\cac_1^\ga ] \, \lln t-s\rrn^{2\ga}\\
& \leq & c_{x,\si} \lln t-s\rrn^\ga\lp 1+\cn[y;\cac_1^\ga]+\cn[\yti; \cac_1^\ga] \rp  \cn[y-\yti; \cac_1^\ga] T_0^\ga,
\eean
which gives $\cn[J^{1,2}; \cac_2^\ga] \leq c_{x,\si} \lp 1+\cn[y;\cac_1^\ga]+\cn[\yti; \cac_1^\ga] \rp  \cn[y-\yti; \cac_1^\ga] T_0^\ga$.

\smallskip

\noindent
{\it Case of} $J^2$: By (\ref{contraction}) and (\ref{inegalite-sigma-3}),
\bean
\norm{J_{st}^2 }&\leq & c \, \cn[[\si^t-\si^s](\cy)-[\si^t-\si^s](\cyti); \cac_1^{\ka\ga} ] \, \cn[x;\cac_1^\ga] \, T_0^{\ga(1+\ka)}\\
&\leq & c_{\si,x} \lln t-s \rrn^\ga T_0^{1+\ga\ka} \cn[y-\yti; \cac_1^\ga] \lcl 1+\cn[y; \cac_1^\ga]^\ka+\cn[\yti; \cac_1^\ga]^\ka \rcl,
\eean
or in other words $\cn[J^2;\cac_2^\ga] \leq c_{\si,x}  T_0^{1+\ga\ka} \cn[y-\yti; \cac_1^\ga] \lcl 1+\cn[y; \cac_1^\ga]^\ka+\cn[\yti; \cac_1^\ga]^\ka \rcl$.

\

Therefore, $\cn[z-\zti; \cac_1^\ga] \leq c_{\si,x} T_0^\ga\, \cn[y-\yti; \cac_1^\ga ] \lcl 1+A_{T_0} \rcl$. Since the radius $A_{T_0}$ decreases as $T_0$ tends to $0$, we can choose a sufficiently small time $T_0 \in (0,T_1]$ such that the application $\Gamma$, restricted to the (stable) ball $B_{T_0,a}^{A_{T_0}}$, is a strict contraction. Hence the existence and uniqueness of a fixed point in this set.

\end{proof}

The next proposition summarizes our considerations in order to get the global existence and uniqueness for solution to equation (\ref{eq:volterra-young}):
\begin{proposition}\label{extension-cas-young}
Under the hypothesis of Theorem \ref{theo:young-case}, the local solution $y^{(1)}$ defined by the previous proposition can be extended to a global and unique solution in $\cac_1^\ga([0,T];\R^d)$.
\end{proposition}

\begin{proof}
In fact, we are going to show the existence of a small $\ep >0$, which \textbf{shall not depend on $y^{(1)}$}, such that $y^{(1)}$ can be extended to a solution on $[0,T_0+\ep]$. The conclusion then follows by a simple iteration argument.

\smallskip

To this end, let us introduce the application $\Gamma$ defined for any $z\in \cac_1^\ga([0,T_0+\ep])$ such that $z_{|[0,T_0]}=y^{(1)}$ as
$$\hz_t=\Gamma(z)_t=\begin{cases}
\yun_t & \mbox{if} \ t\in [0,T_0]\\
a+\cj_{0t}(\si^t(\cz) \, dx) & \mbox{if} \ t\in [T_0,T_0+\ep]
\end{cases}. $$
Just as in the previous proof, we are looking for a fixed point of $\Gamma$.

\smallskip

\noindent
{\it Step 1: Invariance of a ball.} In order to estimate $\cn[\hz; \cac_1^\ga([0,T_0+\ep])]$, let us consider the three cases ($s,t \in [0,T_0]$), ($s,t \in [T_0,T_0+\ep]$) and ($s \leq T_0 \leq t \leq T_0+\ep$).

\smallskip

In the first case, we simply have $\cn[\hz; \cac_1^\ga([0,T_0])] \leq \cn[\yun; \cac_1^\ga([0,T_0])]$. Consider the second case $s, t \in [T_0,T_0+\ep]$, and decompose $(\der \hz)_{st}$ as above, that is $(\der \hz)_{st}=I^{1,1}_{st}+I^{1,2}_{st}+I^{2,1}_{st}+I^{2,2}_{st}$, with
$$I^{1,1}_{st}=\si^t(\cz_s) \, (\der x)_{st} \quad , \quad I_{st}^{1,2}=\laa_{st}(\der(\si^t(\cz)) \, \der x),$$
$$I_{st}^{2,1}=[\si^t-\si^s](\cz_0) \, (\der x)_{0s} \quad , \quad I_{st}^{2,2}=\laa_{0s}(\der([\si^t-\si^s](\cz)) \, \der x).$$
Let us now bound each of these terms: first, owing to (\ref{contraction}) and (\ref{inegalite-sigma-0}), $I_{st}^{1,2}$ can be estimated as follows:
\bean
\norm{I_{st}^{1,2}} &\leq & c \, \cn[\si^t(\cz); \cac_1^\ga([0,T_0+\ep])] \, \cn[x; \cac_1^\ga] \lln t-s \rrn^{2\ga}\\
&\leq & c_{\si,x} \lcl 1+\cn[z;\cac_1^\ga([0,T_0+\ep])]\rcl \lln t-s\rrn^{2\ga}.
\eean 
It is thus readily checked that $\cn[I^{1,2}; \cac_2^\ga([T_0,T_0+\ep])] \leq c_{\si,x} \, \ep^\ga \lcl 1+\cn[z;\cac_1^\ga([0,T_0+\ep])] \rcl$.
Thanks to (\ref{contraction}) and (\ref{inegalite-sigma-1}), we also have the following bound for $I^{2,2}_{st}$: 
\bean
\norm{I_{st}^{2,2}} &\leq & c \, \cn[ [\si^t-\si^s](\cz); \cac_1^\ga([0,T_0+\ep])] \, \cn[x; \cac_1^\ga] \, T^{2\ga}\\
&\leq & c_{\si,x} \lln t-s\rrn  \lcl 1+\cn[z;\cac_1^\ga([0,T_0+\ep])] \rcl,
\eean
which gives $\cn[I^{2,2}; \cac_2^\ga([T_0,T_0+\ep])] \leq c_{\si,x} \, \ep^{1-\ga} \lcl 1+\cn[z;\cac_1^\ga([0,T_0+\ep])] \rcl$.
Since trivially $\cn[I^{i,1};\cac_2^\ga([T_0,T_0+\ep])] \leq c_{\si,x}$ for $i=1,2$, we get
$$\cn[\hz;\cac_1^\ga([T_0,T_0+\ep])] \leq c_{\si,x} \lcl 1+\ep^{1-\ga} \cn[z;\cac_1^\ga([0,T_0+\ep])] \rcl.$$
Finally, let us treat the third case $0\leq s \leq T_0 \leq t\leq T_0+\ep$: write
\bean
\lefteqn{\norm{(\der \hz)_{st}}=\norm{(\der \hz)_{sT_0}+(\der \hz)_{T_0t}}}\\
&\leq & \cn[\yun; \cac_1^\ga([0,T_0])] \lln T_0-s\rrn^\ga+\cn[\hz; \cac_1^\ga([T_0,T_0+\ep])] \lln t-T_0\rrn^\ga\\
&\leq &  \lcl \cn[\yun; \cac_1^\ga([0,T_0])]+\cn[\hz; \cac_1^\ga([T_0,T_0+\ep])] \rcl \lln t-s \rrn^\ga.
\eean

\smallskip

Putting together the three cases we have just studied, the following bound is obtained for $\hz$ on the whole interval $[0,T_0+\ep])]$:
$$\cn[\hz;\cac_1^\ga([0,T_0+\ep])] \leq c_{\si,x}^1 \lcl 1+\cn[\yun; \cac_1^\ga([0,T_0])]+\ep^{1-\ga} \cn[z;\cac_1^\ga([0,T_0+\ep])] \rcl.$$
Therefore, set 
$$\ep=(2 c_{\si,x}^1)^{-1/(1-\ga)} \, \mbox{($\ep$ does not depend on $\yun$) and} \ N_1=2c_{\si,x}^1 \lcl 1+\cn[\yun; \cac_1^\ga([0,T_0])]\rcl,$$ so that if $\cn[z;\cac_1^\ga([0,T_0+\ep])]\leq N_1$, then $\cn[\hz;\cac_1^\ga([0,T_0+\ep])] \leq \frac{N_1}{2}+\frac{N_1}{2} =N_1$. In other words,  we have found that the ball
$$B_{\yun, T_0,\ep}^{N_1}=\lcl z\in \cac_1^\ga([0,T_0+\ep]): \ z_{|[0,T_0]}=\yun, \ \cn[z;\cac_1^\ga([0,T_0+\ep])] \leq N_1 \rcl$$
is invariant by $\Gamma$.

\smallskip

\noindent
{\it Step 2: Contraction property.} This second step consists in finding a small $\eta \in (0,\ep]$ such that the previous application $\Gamma$ (adapted to $[0,T_0+\eta]$) satisfies a contraction property when restricted to some (invariant) ball.

\smallskip

Let $\zun,\zde \in B_{\yun, T_0,\eta}^{N_1}$ and set $\hzun=\Gamma(\zun)$, $\hzde=\Gamma(\zde)$. Of course, since $\hzun$ and $\hzde$ share the same initial condition on $[0,T_0]$, we have $\cn[\hzun-\hzde; \cac_1^\ga([0,T_0+\eta])]=\cn[\hzun-\hzde; \cac_1^\ga([T_0,T_0+\eta])]$. Let then $T_0 \leq s <t \leq T_0+\eta$ and as in the proof of Proposition \ref{prop:3.6}, use the decomposition $\der(\hzun-\hzde)_{st}=J_{st}^{1,1}+J_{st}^{1,2}+J_{st}^2$, where
$$J_{st}^{1,1}=( \si^t(\czun_s)-\si^t(\czde_s) ) \, (\der x)_{st} \quad , \quad J_{st}^{1,2}=\laa_{st}(\der(\si^t(\czun)-\si^t(\czde)) \, \der x),$$
$$J_{st}^{2}=\laa_{0s}( \der([\si^t-\si^s](\czun)-[\si^t-\si^s](\czde)) \, \der x).$$
We will bound again each of these terms separately: for $J^{1,1}$, we have 
$$\norm{J^{1,1}_{st}} \leq \norm{D\si}_\infty \Vert \zun_s-\zde_s \Vert \cn[x;\cac_1^\ga] \lln t-s \rrn^\ga.
$$ 
But
$$\norm{\zun_s-\zde_s}=\norm{[\zun_s-\zde_s]-[\zun_{T_0}-\zde_{T_0}]} \leq \cn[\zun-\zde;\cac_1^\ga([0,T_0+\eta])] \, \eta^\ga,$$
and so
\begin{equation}\label{demo:majo-i-1-1}
\cn[J^{1,1}; \cac_2^\ga([T_0,T_0+\eta])] \leq c_{x,\si} \, \eta^\ga \cn[\zun-\zde;\cac_1^\ga([0,T_0+\eta])].
\end{equation}
The term $J_{st}^{1,2}$ can be estimated as follows: by (\ref{contraction}) and (\ref{inegalite-sigma-2}), 
\bean
\norm{J_{st}^{1,2}} &\leq & c \, \cn[\si^t(\czun)-\si^t(\czde);\cac_1^\ga([0,T_0+\eta])] \cn[x;\cac_1^\ga] \lln t-s\rrn^{2\ga}\\
&\leq & c_{\si,x} \lln t-s\rrn^\ga \eta^\ga \lcl 1+2N_1 \rcl \cn[\zun-\zde; \cac_1^\ga([0,T_0+\eta])].
\eean
Finally, according to  (\ref{contraction}) and (\ref{inegalite-sigma-3}), we have:
\bean
\norm{J_{st}^2} &\leq & c \, \cn[[\si^t-\si^s](\czun)-[\si^t-\si^s](\czde); \cac_1^{\ka\ga}([0,T_0+\eta])] \cn[x;\cac_1^\ga] \, T^{\ga(1+\ka)}\\
&\leq & c_{\si,x} \lln t-s \rrn^\ga \eta^{1-\ga} \lcl 1+2N_1^\ka \rcl \cn[\zun-\zde;\cac_1^\ga([0,T_0+\eta])].
\eean
As a result, putting together the bounds on $J_{st}^{1,1}$, $J_{st}^{1,2}$ and $J_{st}^{2}$, we end up with:
$$\cn[\hzun-\hzde;\cac_1^\ga([0,T_0+\eta])] \leq c_{\si,x}^1 \eta^{1-\ga} \lcl 1+N_1^\ka+N_1 \rcl \cn[\zun-\zde;\cac_1^\ga([0,T_0+\eta])].$$

\smallskip

We can now pick $\eta \in (0,\ep]$ such that $c_{\si,x}^1 \eta^{1-\ga} \lcl 1+N_1^\ka+N_1 \rcl \leq \frac{1}{2}$, and the application $\Gamma$ becomes a strict contraction on $B_{\yun, T_0,,\eta}^{N_1}$. It is easy to check (see Lemma \ref{lem:stable-ball} below) that $B_{\yun, T_0,\eta}^{N_1}$ is invariant by $\Gamma$ too, hence the existence and uniqueness of a fixed point in this set, denoted by $y^{(1),\eta}$.

\smallskip

Notice now that the arguments leading to uniqueness remain true on the (stable) ball
$$\lcl z\in \cac_1^\ga([0,T_0+2\eta]): \ z_{|[0,T_0+\eta]}=y^{(1),\eta}, \ \cn[z; \cac_1^\ga([0,T_0+2\eta])] \leq N_1 \rcl.$$
For instance, to establish the equivalent of relation (\ref{demo:majo-i-1-1}) on this extended interval, notice that if $s\in [T_0+\eta,T_0+2\eta]$, 
$$\norm{\zun_s-\zde_s} =\norm{[\zun_s-\zde_s]-[\zun_{T_0+\eta}-\zde_{T_0+\eta}] }\leq \cn[\zun-\zde;\cac_1^\ga([0,T_0+2\eta])] \, \eta^\ga.$$
This enables to extend $y^{(1),\eta}$ into a solution $y^{(1),2\eta}$ on $[0,T_0+2\eta]$, and then $y^{(1),3\eta}$ on $[0,T_0+3\eta]$, ... until $[0,T_0+\eta]$ is covered, as we wished.

\end{proof}

\begin{lemma}\label{lem:stable-ball}
With the notations of the previous proof, the sets 
$$\lcl z\in \cac_1^\ga([0,T_0+l\eta]): \ z_{|[0,T_0+(l-1)\eta]}=y^{(1),(l-1)\eta}, \ \cn[z;\cac_1^\ga([0,T_0+l\eta])] \leq N_1 \rcl$$
are invariant by $\Gamma$.
\end{lemma}
\begin{proof}
If $z$ belongs to such a ball, set
$$\zti_t=\begin{cases}
z_t & \mbox{if} \ t \in [0,T_0+l\eta]\\
z_{T_0+l\eta} & \mbox{if} \ t\in [T_0+l\eta,T_0+\ep]
\end{cases}.$$
Clearly, $\zti \in B_{\yun, T_0,\ep}^{N_1}$, so that, thanks to the first step of the previous proof, $\Gamma(\zti) \in B_{\yun, T_0,\ep}^{N_1}$. Now, since $y^{(1),(l-1)\eta}$ is a solution on $[0,T_0+(l-1)\eta]$, we have $\Gamma(\zti)_{|[0,T_0+(l-1)\eta]}=y^{(1),(l-1)\eta}$, which means that $\Gamma(\zti)$ is an extension of $\Gamma(z)$ and as a result
$$\cn[\Gamma(z); \cac_1^\ga([0,T_0+l\eta])] \leq \cn[\Gamma(\zti); \cac_1^\ga([0,T_0+\ep])] \leq N_1.$$

\end{proof}


\section{The Young singular case}
\label{sec:singular-case}

This section is devoted to the study of a particular case of Equation (\ref{eq:volterra}), when the coefficient $\si$ admits a singularity in $(t,u)$ on the diagonal. Namely, we shall consider an equation of the form
\begin{equation}\label{eq:volterra-singular}
y_t=a+\int_0^t (t-u)^{-\al} \psi(y_u) \, dx_u,
\end{equation}
with $\psi: \R^d \to \R^{d,n}$ a sufficiently regular mapping and $x \in \cac_1^\ga([0,T];\R^n)$, for some $\ga$ and $\al$ to be precised. Thus, the application $\si$ appearing in (\ref{eq:volterra}) tends here to explode when approaching the diagonal
$$D \times \R^d =\lcl (t,t,y), \ t\in [0,T],y\in \R^d \rcl.$$
This singularity prevents us from directly applying the algebraic formalism introduced at section \ref{sec:alg-integration} in order to define the integral $\int_0^t (t-u)^{-\al} \psi(y_u) \, dx_u$ above. However, as in Section \ref{sec:young-case}, we shall see that this latter integral can still be defined thanks to a slight extension of Young's interpretation, insofar as the integral will simply be seen as the limit of the associated Riemann sums. In other words, we will be able to set
\begin{equation}\label{integrale-young-singu}
\int_s^t (t-u)^{-\al} \psi(y_u) \, dx_u =\lim_{k \to \infty} \sum_{\Delta_k([s,t))} (t-t_i)^{-\al}\psi(y_{t_i}) \, (\der x)_{t_i t_{i+1}},
\end{equation}
where $\Delta_k([s,t))=\lcl s=t_0 <t_1 < \ldots < t_k <t \rcl$ is any sequence of partitions whose meshes tend to $0$, and where $t_k \to t$. In this context, Theorem \ref{theo:singular-case} is quite close to Theorem \ref{theo:young-case}.

\begin{remark}
The tedious calculations to come will give us an idea of how the $\Lambda$-formalism used in the previous sections makes the writing more fluent (when it can be applied), by avoiding the often cumbersome study of Riemann sums.
\end{remark}

\subsection{Young singular integrals}
This section deals with a rigorous definition of integrals like (\ref{integrale-young-singu}). A first technical lemma in this direction is then the following:
\begin{lemma}\label{lem:4.2}
Let $a<b$, $f \in \cac^{1,\textbf{\textit{b}}}([a,b];\R),g \in \cac_1^{\la_1}([a,b];\R^{d,n}), h \in \cac_1^{\la_2}([a,b];\R^n)$ with $\la_1+\la_2 >1$. Then
$$\int_a^b d(fg)_u \, h_u=\int_a^b df_u \, g_uh_u+\int_a^b dg_u \, f_u h_u,$$
the three integrals being understood in the Young sense.
\end{lemma}

\begin{proof}
Consider a partition $\Delta =\{a=t_0<\cdots<t_n=b\}$, with mesh $|\Delta|$, and use the decomposition
$$\sum_i  \der(fg)_{t_i t_{i+1}} h_{t_i}=\sum_i (\der f)_{t_it_{i+1}}g_{t_i}h_{t_i}+\sum_i (\der g)_{t_i t_{i+1}}f_{t_i}h_{t_i}
+\sum_i (\der f)_{t_i t_{i+1}}(\der g)_{t_i t_{i+1}}h_{t_i}.$$
Notice then that
$$\norm{ \sum_i (\der f)_{t_i t_{i+1}}(\der g)_{t_i t_{i+1}}h_{t_i} } \leq \cn[f;\cac^{1,\textbf{\textit{b}}}] \cn[g;\cac_1^{\la_1}] \lln \Delta \rrn^{\la_1} \cn[h;\cac_1^0] \lln b-a \rrn,$$
which tends to $0$ as $|\Delta | \to 0$. The proof is thus easily finished.

\end{proof}

\begin{lemma}\label{lem:4.3}
If $\ga > \al$ and $\psi \in \cac^{1,\textbf{\textit{b}}}(\R^d;\R^{d,n})$, then for any $\ka$ such that $(\ga-\al)+\ka >1$ and any $y \in \cac_1^\ka([0,T];\R^d)$, the integral $I_{st}:=\int_s^t (t-u)^{-\al}\psi(y_u) \, dx_u$ exists in the Young sense. More specifically, for any $0\le s< t \le T$ and $0<\ep<t-s$, set $I_{st}^\ep:=\int_s^{t-\ep} (t-u)^{-\al} \psi(y_u) \, dx_u$, defined in the Young sense of Proposition \ref{prop:3.1}. Then $I_{st}^\ep$ converges to a quantity, which is denoted again by $\int_s^{t} (t-u)^{-\al} \psi(y_u) \, dx_u$.
\end{lemma}

\begin{proof}

Let $\ep >0$. If $u,v \in [s,t-\ep]$,
\bean
\norm{ \frac{\psi(y_v)}{(t-v)^\al}-\frac{\psi(y_u)}{(t-u)^\al} } &\leq & \norm{\psi}_\infty \lln \frac{1}{(t-v)^\al}-\frac{1}{(t-u)^\al}\rrn+\lln \frac{1}{(t-u)^\al}\rrn \norm{ \psi(y_v)-\psi(y_u)}\\
&\leq & \norm{\psi}_\infty \frac{\al}{\ep^{\al+1}} \lln v-u\rrn+\frac{1}{\ep^\al}\norm{\psi'}_\infty \cn[y;\cac_1^\ka([0,T])] \lln v-u\rrn^\ka,
\eean
hence $ u \mapsto \frac{\psi(y_u)}{(t-u)^\al} \in \cac^\ka([s,t-\ep])$ and since $\ka+\ga >1$, the integral $I_{st}^\ep$ is well-defined in the Young sense of Proposition \ref{prop:3.1}. We will now study the convergence of $I_{st}^{\ep}$ when $\ep\to 0$. 

\smallskip

It is easily checked from relation (\ref{integrale-young-exacte}) that one is allowed to perform a integration by parts in $I_{st}^{\ep}$, in order to deduce
\bean
I_{st}^{\ep}&=&\int_s^{t-\ep}(t-u)^{-\al} \psi(y_u)\, dx_u \ = \ \int_s^{t-\ep} (t-u)^{-\al}\psi(y_u)\, d(x_u-x_t)\\
&=& \frac{\psi(y_{t-\ep})}{\ep^\al} (x_{t-\ep}-x_t)+\frac{\psi(y_s)}{(t-s)^\al}(x_t-x_s)+\int_s^{t-\ep} d\lp \frac{\psi(y_u)}{(t-u)^\al} \rp\, (x_t-x_u)  \\
&:=&I_{st}^{\ep,1}+I_{st}^{\ep,2}+I_{st}^{\ep,3}.
\eean
Let us analyze now the three terms we have obtained: since
$$\left\Vert \frac{\psi(y_{t-\ep})}{\ep^\al} (x_{t-\ep}-x_t) \right\Vert \leq \norm{\psi}_\infty \cn[x;\cac_1^\ga]  \ep^{\ga-\al},$$
it is readily checked that  $I_{st}^{\ep,1}\to 0$ as $\ep\to 0$. In order to treat the term $I_{st}^{\ep,3}$ observe that, according to Lemma \ref{lem:4.2}, we have
\begin{multline}\label{int-par-parties-young}
I_{st}^{\ep,3}=\int_s^{t-\ep} d\lp \frac{\psi(y_u)}{(t-u)^\al} \rp(x_t-x_u) \\
=\int_s^{t-\ep}  d(\psi(y_u))\frac{(x_t-x_u)}{(t-u)^{\al}}+\al \int_s^{t-\ep} \frac{du}{(t-u)^{\al+1}}\psi(y_u) (x_t-x_u):=I_{st}^{\ep,3,1}+I_{st}^{\ep,3,2}.
\end{multline}
Notice then that
$$\left\Vert \frac{\psi(y_u)}{(t-u)^{\al+1}}(x_t-x_u) \right\Vert \leq \norm{\psi}_\infty \cn[x;\cac_1^\ga] \frac{1}{\lln t-u\rrn^{1-(\ga-\al)}},$$
and thus $u \mapsto  \frac{\psi(y_u)}{(t-u)^{\al+1}}(x_t-x_u)$ is (Lebesgue-)integrable in $t$. This trivially yields  the convergence of $I_{st}^{\ep,3,2}$ as $\ep\to 0$. As for the first term $I_{st}^{\ep,3,1}$ in (\ref{int-par-parties-young}), we know that $u\mapsto \psi(y_u) \in \cac_1^\ka$. In order to study the convergence of $I_{st}^{\ep,3,2}$, it only remains to prove that the application $\vp:[s,t) \to \R^n, \, u \mapsto  \frac{(x_t-x_u)}{(t-u)^{\al}}$, continuously extended by $0$ in $t$, belongs to $\cac_1^\rho([s,t])$, for some $\rho >0$ satisfying  $\rho +\ka >1$.

\smallskip

However, if $0 <u <v <t$,
\bean
\lefteqn{\norm{\vp_v-\vp_u }}\\
&\leq & \norm{ x_t-x_v } \lln (t-v)^{-\al}-(t-u)^{-\al}\rrn +\lln (t-u)^{-\al}\rrn \norm{ x_t-x_v-(x_t-x_u)}\\
&\leq & \cn[x;\cac_1^\ga] \lln t-v\rrn^\ga \lp \frac{1}{\lln t-v\rrn^\al} \rp^{1-(\ga-\al)}\lp \al\frac{\lln v-u\rrn}{\lln t-v\rrn^{\al+1}} \rp^{\ga-\al}+\frac{1}{\lln v-u\rrn^\al} \cn[x;\cac_1^\ga] \lln v-u\rrn^\ga\\
&\leq & c \, \cn[x;\cac_1^\ga] \lln v-u \rrn^{\ga-\al}+\cn[x;\cac_1^\ga] \lln v-u\rrn^{\ga-\al}, 
\eean
while if $u <v=t$, as $\vp_t=0$,
$$\norm{ \vp_v-\vp_u } = \frac{\norm{x_v-x_u}}{\lln (v-u)^\al\rrn} \leq \cn[x;\cac_1^\ga] \lln v-u \rrn^{\ga-\al}.$$
Thus, $\vp \in \cac_1^{\ga-\al}([0,t])$, which achieves the proof since, by hypothesis, $(\ga-\al)+\ka >1$.
\end{proof}

It is also important to control the Hölder continuity of the singular Young integral defined above. Before we turn to this task, let us quote an elementary estimate for further use:
\begin{lemma}
Let $0<s<t\leq T$. For any $\be \in [0,1]$, there exists a constant $c_\be$ such that for any $u\in (0,s)$, 
\begin{equation}\label{estimation-f-s-t}
\lln (t-u)^{-\al}-(s-u)^{-\al} \rrn \leq c_\be \lln s-u\rrn^{-\al-\be} \lln t-s \rrn^\be.
\end{equation}
\end{lemma}
Then our regularity result is the following:
\begin{proposition}\label{prop:4.5}
Under the same assumptions as in Lemma \ref{lem:4.3}, set $z_t=I_{0t}$ for all $t\in[0,T]$. Then, for any $T_0\le T$, the path $z$ is an element of $\cac_1^\ka([0,T_0])$, and the following estimate holds true:
$$
\cn[z;\cac_1^\ka([0,T_0])] \leq c_{\psi,x} T_0^{\ga-\al-\ka}  \lcl 1+\cn[y;\cac_1^\ka([0,T_0])]\rcl.
$$
\end{proposition}

\begin{proof}
We rely on the decomposition $(\der z)_{st}=I_{st}+II_{st}$, with
\begin{eqnarray} \label{decompo}
I_{st}=\int_s^t (t-u)^{-\al} \psi(y_u)dx_u \quad \mbox{and} \quad II_{st}=\int_0^s \lc(t-u)^{-\al}-(s-u)^{-\al}\rc\psi(y_u) \, dx_u.
\end{eqnarray}
Notice that the term $I$ is exactly the one introduced at Lemma \ref{lem:4.3}. Let us now bound each of these terms.

\smallskip

\noindent
\textit{Case of $I$:} It is easily seen that $I$ can also be obtained thanks to the following approximation sequence: for $n\ge 1$, set
$$J_n =\sum_{i=0}^{2^n-1} (t-s_n^i)^{-\al} \psi(y_{s_n^i}) (\der x)_{s_n^i,s_n^{i+1}}, \quad \mbox{where} \ s_n^i=s+\frac{i(t-s)}{2^n} .$$
Then $I_{st}$ is obtained as $\lim_{n\to\infty}J_n$. Moreover, it is readily checked that
\begin{eqnarray}
J_{n+1}-J_n &=& \sum_{i=0}^{2^n-1} \lc (t-s_{n+1}^{2i+1})^{-\al}\psi(y_{s_{n+1}^{2i+1}})-(t-s_{n+1}^{2i})^{-\al}\psi(y_{s_{n+1}^{2i}})\rc(\der x)_{s_{n+1}^{2i+1},s_{n+1}^{2+2}} \nonumber\\
&=& \sum_{i=0}^{2^n-1} \lc (t-s_{n+1}^{2i+1})^{-\al}-(t-s_{n+1}^{2i})^{-\al}\rc\psi(y_{s_{n+1}^{2i+1}})(\der x)_{s_{n+1}^{2i+1},s_{n+1}^{2i+2}} \nonumber\\
& & \hspace{3cm} +\sum_{i=0}^{2^n-1} (t-s_{n+1}^{2i})^{-\al}\lc\psi(y_{s_{n+1}^{2i+1}})-\psi(y_{s_{n+1}^{2i}})\rc(\der x)_{s_{n+1}^{2i+1},s_{n+1}^{2i+2}} \nonumber\\
&:=& A+B. \label{decompojn}
\end{eqnarray}

But
$$\norm{ A } \leq \|\psi\|_\infty \cn[x;\cac_1^\ga] \frac{\lln t-s \rrn^\ga}{(2^{n+1})^\ga} \sum_{i=0}^{2^n-1} \lln (t-s_{n+1}^{2i+1})^{-\al}-(t-s_{n+1}^{2i})^{-\al} \rrn ,$$
and a telescopic sum kind of argument shows that
\begin{align}\label{majo-somme-f-t}
&\sum_{i=0}^{2^n-1} \lln (t-s_{n+1}^{2i+1})^{-\al}-(t-s_{n+1}^{2i})^{-\al} \rrn \\
&= (t-s)^{-\al} \sum_{i=0}^{2^n-1} \lcl \lp 1-\frac{2i+1}{2^{n+1}} \rp^{-\al}-\lp 1-\frac{2i}{2^{n+1}} \rp^{-\al} \rcl  \nonumber\\
&\leq  (t-s)^{-\al} \lp 1-\frac{2^{n+1}-1}{2^{n+1}} \rp^{-\al} \ \leq \ (t-s)^{-\al} (2^{n+1})^\al. \nonumber
\end{align}
Hence
\begin{eqnarray}\label{termeA}
\norm{ A } \leq c_{\psi,x} \lln t-s \rrn^{\ga-\al} \lp \frac{1}{2^{\ga-\al}} \rp^{n+1} \leq c_{\psi,x} \lln t-s \rrn^{\ka} T_0^{\ga-\al-\ka}\lp \frac{1}{2^{\ga-\al}} \rp^{n+1}.
\end{eqnarray}
As for $B$, the following bound holds true:
$$\norm{ B } \leq \lp \sum_{i=0}^{2^n-1} (t-s_{n+1}^{2i})^{-\al} \rp \|\psi'\|_\infty \cn[y;\cac_1^\ka] \frac{\lln t-s\rrn^{\ka}}{(2^{n+1})^\ka} \cn[x;\cac_1^\ga] \frac{\lln t-s\rrn^\ga}{(2^{n+1})^\ga} ,$$
with
\begin{multline*}
\sum_{i=0}^{2^n-1} (t-s_{n+1}^{2i})^{-\al}=(t-s)^{-\al} \sum_{i=0}^{2^n-1} \lp 1-\frac{2i}{2^{n+1}} \rp^{-\al} 
\leq \frac{2^{n+1}}{(t-s)^{\al}} \int_0^1 \frac{du}{(1-u)^\al} \\
\leq \frac{2^{n+1}}{1-\al} (t-s)^{-\al},
\end{multline*}
and accordingly
\begin{eqnarray}
\norm{ B } &\leq & c_{\psi,x} \cn[y;\cac_1^\ka] \lln t-s\rrn^{\ka+\ga-\al}   \lp \frac{1}{2^{\ka+\ga-1}} \rp^{n+1}  \label{termeB}\\
&\leq & c_{\psi,x} \lln t-s \rrn^\ka \cn[y;\cac_1^\ka] T_0^{\ga-\al}\lp \frac{1}{2^{\ka+\ga-1}} \rp^{n+1}   .\nonumber
\end{eqnarray}

Going back to (\ref{decompojn}) and putting together our estimates for $A$ and $B$, we get
$$\norm{J_{n+1}-J_n } \leq T_0^{\ga-\al-\ka} \lln t-s \rrn^\ka \lcl 1+\cn[y;\cac_1^\ka]\rcl \, v_n,$$
where $v_n$ is the general term of a converging series. Now, write
$J_N=J_0+\sum_{n=0}^{N-1} (J_{n+1}-J_n)$,
so that, by letting $n$ tend to infinity, we obtain
$$\left\Vert \int_s^t (t-u)^{-\al} \psi(y_u)dx_u \right\Vert \leq \norm{ J_0}  +T_0^{\ga-\al-\ka} \lln t-s \rrn^\ka \lcl 1+\cn[y;\cac_1^\ka]\rcl .$$

It only remains to notice that
\begin{eqnarray} \label{j0}
\norm{ J_0} =\norm{ (t-s)^{-\al} \psi(y_s)(\der x)_{st}} \leq \|\psi\|_\infty \cn[x;\cac_1^\ga] \lln t-s\rrn^{\ga-\al} \leq c_{\psi,x} \lln t-s\rrn^\ka T_0^{\ga-\al-\ka}
\end{eqnarray}
to conclude
$$\norm{I_{st}} \leq T_0^{\ga-\al-\ka} \lln t-s \rrn^\ka \lcl 1+\cn[y;\cac_1^\ka]\rcl.$$

\smallskip

\noindent
\textit{Case of $II$:} We use the same strategy as for $I$, with this time $s_n^i=\frac{is}{2^n}$ and 
$$  J_n=\sum_{i=0}^{2^n-1} f_{s,t}(s_n^i) \psi(y_{s_n^i})  (\der x)_{s_n^i,s_n^{i+1}}, \quad \mbox{where} \ f_{s,t}(u)=\lc(t-u)^{-\al}-(s-u)^{-\al}\rc .
$$
Then
\begin{eqnarray}
J_{n+1}-J_n& = & \sum_{i=0}^{2^n-1} \lcl f_{s,t}(s_{n+1}^{2i+1}) \psi(y_{s_{n+1}^{2i+1}})-f_{s,t}(s_{n+1}^{2i})\psi(y_{s_{n+1}^{2i}}) \rcl (\der x)_{s_{n+1}^{2i+1},s_{n+1}^{2i+2}} \nonumber\\
&=& \sum_{i=0}^{2^n-1}  \lcl f_{s,t}(s_{n+1}^{2i+1})-f_{s,t}(s_{n+1}^{2i}) \rcl \psi(y_{s_{n+1}^{2i+1}})(\der x)_{s_{n+1}^{2i+1},s_{n+1}^{2i+2}} \nonumber\\
& & \hspace{3cm} +\sum_{i=0}^{2^n-1} f_{s,t}(s_{n+1}^{2i}) \lcl \psi(y_{s_{n+1}^{2i+1}})-\psi(y_{s_{n+1}^{2i}}) \rcl (\der x)_{s_{n+1}^{2i+1},s_{n+1}^{2i+2}} \nonumber\\
&:=& D+E. \label{decompojn2}
\end{eqnarray}

\smallskip

To deal with $D$, notice that $u\mapsto f_{s,t}(u)$ is a decreasing function on $[0,s]$, and hence
\begin{equation}\label{sommef}
\sum_{i=0}^{2^n-1} \lln f_{s,t}(s_{n+1}^{2i+1})-f_{s,t}(s_{n+1}^{2i}) \rrn
\leq \sum_{i=0}^{2^{n+1}-1}  \lln f_{s,t}(s_{n+1}^{i+1})-f_{s,t}(s_{n+1}^{i})\rrn \leq \left| f_{s,t}\lp \frac{2^{n+1}-1}{2^{n+1}}s \rp \right|.
\end{equation}
Furthermore, according to our elementary bound (\ref{estimation-f-s-t}) applied with $\be=\ka$, we have $|f_{s,t}\lp \frac{2^{n+1}-1}{2^{n+1}}s \rp | \leq \frac{c}{s^{\al+\ka}} \lln t-s\rrn^\ka (2^{\al+\ka})^{n+1}$, so that
\begin{eqnarray} 
\norm{ D } &\leq& c \|\psi\|_\infty \cn[x;\cac_1^\ga] s^{\ga-\al-\ka} \lln t-s \rrn ^\ka \lp \frac{1}{2^{\ga-\al-\ka}} \rp^{n+1} \nonumber\\
& \leq & c_{\psi,x} \, T_0^{\ga-\al-\ka} \lln t-s \rrn ^\ka \lp \frac{1}{2^{\ga-\al-\ka}} \rp^{n+1}.\label{termeAbis}
\end{eqnarray}
As far as $E$ is concerned, use (\ref{estimation-f-s-t}) with $\be=\ga-\al$ to deduce
\begin{eqnarray}
\norm{ E } &\leq & c\,  \|\psi'\|_\infty \cn[y;\cac_1^\ka] \cn[x;\cac_1^\ga] s^{\ka}\lln t-s \rrn^{\ga-\al} \lp \frac{1}{2^{\ka+\ga}} \rp^{n+1} \sum_{i=0}^{2^n-1} \lp 1-\frac{2i}{2^{n+1}} \rp^{-\ga}  \nonumber\\
& \leq & c_{\psi,x} \cn[y;\cac_1^\ka]  s^{\ka} \lln t-s \rrn^{\ga-\al} \lp \frac{1}{2^{\ka+\ga-1}} \rp^{n+1} \int_0^1 \frac{dx}{(1-x)^{\ga}} \nonumber \\
&\leq &  c_{\psi,x} \cn[y;\cac_1^\ka] \lln t-s \rrn^\ka \lln t-s \rrn^{\ga-\al-\ka} T_0^{\ka} \lp \frac{1}{2^{\ka+\ga-1}}\rp^{n+1} ,\label{terme-E}
\end{eqnarray}
hence
\begin{eqnarray*} 
\norm{ E }  \leq c_{\psi,x} \cn[y;\cac_1^\ka] T_0^{\ga-\al} \lln t-s\rrn^\ka \lp \frac{1}{2^{\ka+\ga-1}} \rp^{n+1}.
\end{eqnarray*}

\smallskip

Just as for $I$, gathering our bounds on $D$ and $E$, we can then assert that
$$ \norm{ \int_0^s \lc(t-u)^{-\al}-(s-u)^{-\al}\rc\psi(y_u)\, dx_u } \leq \norm{ J_0 } +c_{\psi,x}T_0^{\ga-\al-\ka} \lln t-s \rrn^\ka \lcl 1+\cn[y;\cac_1^\ka]\rcl.$$
Since $\lln t^{-\al}-s^{-\al} \rrn \leq c \, s^{-\al-\ka} \lln t-s \rrn^\ka$, the term $J_0$ above can be estimated as:
\begin{eqnarray} \label{j02}
\norm{ J_0 } =\norm{ \lcl t^{-\al}-s^{-\al} \rcl (\der x)_{0s} } \leq  \cn[x;\cac_1^\ga] s^{\ga-\al-\ka} \lln t-s \rrn^\ka,
\end{eqnarray}
so that 
$$\norm{II_{st}}=\norm{ \int_0^s \lc(t-u)^{-\al}-(s-u)^{-\al}\rc \psi(y_u)\, dx_u } \leq c_{\psi,x}T_0^{\ga-\al-\ka} \lln t-s \rrn^\ka \lcl 1+\cn[y;\cac_1^\ka]\rcl.$$

\smallskip

Finally, going back to decomposition (\ref{decompo}), our bounds on $I$ and $II$ yield 
$$\cn[z;\cac_1^\ka] \leq c_{\psi,x}T_0^{\ga-\al-\ka}  (1+\cn[y;\cac_1^\ka]),$$
which was the announced result.

\end{proof}

\subsection{Solving Volterra equations}
Thanks to the considerations of the last section, we can now interpret equation (\ref{eq:volterra-singular}), and especially its integral term, in the sense  given by Lemma \ref{lem:4.3} and Proposition \ref{prop:4.5}. We are now in position to state the main result of this section:
\begin{theorem}\label{theo:singular-case}
Assume that $x \in \cac_1^\ga([0,T];\R^n)$ for some $\ga \in (1/2,1)$, let $\psi$ be a function in  $\cac^{1,\textbf{\textit{b}}}(\R^d; \R^{d,n})$, and $\al \in (0,1/2)$ such that $\ga-\al > 1/2$. Then, for any $\ka \in (1-(\ga-\al); \ga-\al)$, equation (\ref{eq:volterra-singular}) admits a unique solution in $\cac_1^\ka([0,T]; \R^d)$.
\end{theorem}

Fix $\ka \in (1-(\ga-\al),\ga-\al)$. As in Section \ref{sec:young-case}, we shall solve our equation by identifying its solution with the fixed point of the map $\Gamma$ defined, for any $y\in \cac_1^\ka([0,T];\R^d)$, by 
\beq\label{eq:34}
z_t=\Gamma(y)_t=a+\int_0^t (t-u)^{-\al} \psi(y_u) \, dx_u.
\eeq
We divide again our proof into two propositions, dealing respectively with local and global existence and uniqueness for the solution.

\smallskip

\begin{proposition}[Local existence]
Under the hypothesis of Theorem \ref{theo:singular-case}, there exists $T_0 \in (0,T]$ such that Equation (\ref{eq:volterra-singular}) admits a unique solution $\yun$ in $\cac_1^\ka([0,T_0]; \R^d)$.
\end{proposition}

\begin{proof}
Fix a time $T_0 \in (0,T]$ and let $y\in \cac^\ka([0,T_0])$. Define then $z=\Gamma(y)$ as in equation~(\ref{eq:34}).

\smallskip

\noindent
\textit{Step 1: Invariance of a ball}. A simple application of Proposition \ref{prop:4.5} allows to conclude the existence of a stable ball
$$\cb_{a, T_0} =\lcl y \in \cac^\ka([0,T_0]), \ y_0=a, \ \cn[y;\cac_1^\ka] \leq A_{T_0} \rcl$$
for any $T_0$ small enough and $A_{T_0}$ large enough.

\smallskip

\noindent
\textit{Step 2: Contraction property}. Let $y,\yti \in \cb_{a, T_0}$, and set $z=\Gamma(y)$, $\zti=\Gamma(\yti)$. Thus, $\der(z-\zti)_{st}=III_{st}+IV_{st}$, with
\begin{eqnarray} \label{decompocontra}
III_{st}&=&\int_s^t (t-u)^{-\al} \lc \psi(y_u)-\psi(\yti_u) \rc dx_u \\ 
IV_{st}&=&\int_0^s \lc(t-u)^{-\al}-(s-u)^{-\al} \rc \lc\psi(y_u)-\psi(\yti_u)\rc dx_u. \nonumber
\end{eqnarray}
We will now estimate these two terms, according to the same strategy as for Proposition~\ref{prop:4.5}, i.e. invoking approximations by dyadic partitions.

\smallskip

\noindent
\textit{Case of $III$:} Denote
$$s_n^i=s+\frac{i(t-s)}{2^n}, \quad  J_n= \sum_{i=0}^{2^n-1} (t-s_n^i)^{-\al} \lc\psi(y_{s_n^i})-\psi(\yti_{s_n^i})\rc(\der x)_{s_{n}^{i},s_{n}^{i+1}}.$$
Then
\begin{eqnarray}
\lefteqn{J_{n+1}-J_n} \nonumber\\
&=& \sum_{i=0}^{2^n-1} \lcl \lc (t-s_{n+1}^{2i+1})^{-\al}-(t-s_{n+1}^{2i})^{-\al}\rc\lc\psi(y_{s_{n+1}^{2i+1}})-\psi(\yti_{s_{n+1}^{2i+1}}) \rc \rcl (\der x)_{s_{n+1}^{2i+1},s_{n+1}^{2i+2}} \nonumber\\
& & \hspace{0.5cm} +\sum_{i=0}^{2^n-1} \lcl (t-s_{n+1}^{2i})^{-\al} \lc \psi(y_{s_{n+1}^{2i+1}})-\psi(\yti_{s_{n+1}^{2i+1}})-\psi(y_{s_{n+1}^{2i}})+\psi(\yti_{s_{n+1}^{2i}})\rc \rcl (\der x)_{s_{n+1}^{2i+1},s_{n+1}^{2i+2}} \nonumber\\
&:=& F+G. \label{decompojn3}
\end{eqnarray}
For $F$, we have, since $(y-\yti)_0=0$,
$$
\norm{F}
\leq  \cn[x;\cac_1^\ga] \frac{\lln t-s \rrn^\ga}{(2^{n+1})^\ga} \|\psi'\|_\infty \cn[y-\yti;\cac_1^\ka] T_0^\ka \sum_{i=0}^{2^n-1}  |(t-s_{n+1}^{2i+1})^{-\al}-(t-s_{n+1}^{2i})^{-\al} |,
$$
which, thanks to (\ref{majo-somme-f-t}), gives
\begin{eqnarray}\label{relation1}
\norm{ F } 
&\leq& c_{\psi,x} \cn[y-\yti;\cac_1^\ka] \lln t-s \rrn^{\ga-\al-\ka} \lln t-s \rrn^{\ka} \lp \frac{1}{2^{\ga-\al}} \rp^{n+1} T_0^\ka.
\end{eqnarray}
As far as $G$ is concerned, use (\ref{inegalite-sigma-2}) to assert that
\begin{multline*}  
\norm{ \psi(y_{s_{n+1}^{2i+1}})-\psi(\yti_{s_{n+1}^{2i+1}})-\psi(y_{s_{n+1}^{2i}})+\psi(\yti_{s_{n+1}^{2i}}) } \\
\leq c_\psi \lcl 1+\cn[y;\cac_1^\ka]+\cn[\yti;\cac_1^\ka] \rcl \cn[y-\yti;\cac_1^\ka ] \frac{\lln t-s \rrn^\ka}{(2^{n+1})^\ka}.
\end{multline*}
Besides,
$$\sum_{i=0}^{2^n-1} (t-s_{n+1}^{2i})^{-\al} \leq \frac{2^{n+1}}{(t-s)^\al} \int_0^1 \frac{du}{(1-u)^\al} ,$$
so that
\begin{eqnarray} \label{relation2}
\norm{G } \leq c_{\psi,x} \cn[y-\yti;\cac_1^\ka] \lcl 1+2A_{T_0} \rcl \lln t-s \rrn^\ka \lp \frac{1}{2^{\ga+\ka-1}} \rp^{n+1} \lln t-s \rrn^{\ga-\ka}.
\end{eqnarray}
Now, relations (\ref{relation1}) and (\ref{relation2}) entail
$$
\norm{III_{st}}\le \norm{J_0}+\sum_{i=0}^{\infty} \norm{J_{n+1}-J_{n}}
\le \norm{J_0}+c_{\psi,x} T_0^{\ga-\al}\lcl 1+2A_{T_0}\rcl \cn[y-\yti;\cac_1^\ka] \lln t-s\rrn^\ka .
$$
Furthermore, we have
\begin{eqnarray} \label{j03}
\norm{J_0} &=&\norm{ (t-s)^{-\al} \lc \psi(y_s)-\psi(\yti_s) \rc (\der x)_{st} }\\
& \leq& \lln t-s \rrn^\ka \lln t-s\rrn^{\ga-\al-\ka} \cn[x;\cac_1^\ga] \|D\psi\|_\infty \cn[y-\yti;\cac_1^\ka] s^\ka, \nonumber  \\
& \leq& c_{\psi,x} T_0^{\ga-\al-\ka} \lln t-s\rrn^\ka \cn[y-\yti ;\cac_1^\ka ] \nonumber
\end{eqnarray}
which finally yields
$$\norm{III_{st}} \leq c_{\psi,x}  T_0^{\ga-\al-\ka}\lcl 1+2A_{T_0}\rcl \cn[y-\yti;\cac_1^\ka] \lln t-s\rrn^\ka .$$

\smallskip

\noindent
\textit{Case of $IV$:} In this case, the approximating sequence is defined by:
$$s_n^i=\frac{is}{2^n}, \quad J_n=\sum_{i=0}^{2^n-1} f_{s,t}(s_n^i) \lc \psi(y_{s_n^i})-\psi(\yti_{s_n^i})\rc (\der x)_{s_n^i, s_n^{i+1}}.$$
Hence, the difference $J_{n+1}-J_n$ can be decomposed into:
\begin{eqnarray}
\lefteqn{J_{n+1}-J_n} \nonumber\\
&=& \sum_{i=0}^{2^n-1} \lcl \lc f_{s,t}(s_{n+1}^{2i+1})-f_{s,t}(s_{n+1}^{2i})\rc \lc \psi(y_{s_{n+1}^{2i+1}})-\psi(\yti_{s_{n+1}^{2i+1}}) \rc \rcl(\der x)_{s_{n+1}^{2i+1},s_{n+1}^{2i+2}}  \nonumber\\
& & \hspace{0.3cm}+ \sum_{i=0}^{2^n-1} \lcl f_{s,t}(s_{n+1}^{2i}) \lc \psi(y_{s_{n+1}^{2i+1}})-\psi(\yti_{s_{n+1}^{2i+1}})-\psi(y_{s_{n+1}^{2i}})+\psi(\yti_{s_{n+1}^{2i}})\rc \rcl (\der x)_{s_{n+1}^{2i+1},s_{n+1}^{2i+2}} \nonumber \\
&:=& H+K. \label{decompo-H-J}
\end{eqnarray}
In order to bound these two terms, let us introduce first some $\la \in (\ka,\ga-\al)$. From (\ref{sommef}), and invoking (\ref{estimation-f-s-t}) with $\be=\la$, we obtain
$$\sum_{i=0}^{2^n-1} \lln f_{s,t}(s_{n+1}^{2i+1})-f_{s,t}(s_{n+1}^{2i}) \rrn \leq c \lln t-s \rrn^\la \frac{(2^{\al+\la})^{n+1}}{s^{\al+\la}},$$
while $\norm{\psi(y_{s_{n+1}^{2i+1}})-\psi(\yti_{s_{n+1}^{2i+1}}) } \leq \norm{\psi'}_\infty \cn[y-\yti;\cac_1^\ka] \, s^\ka$, and so
\begin{eqnarray} 
\norm{ H } &\leq& c_{\psi,x} \lln t-s \rrn^\ka \lln t-s \rrn^{\la-\ka}\cn[y-\yti;\cac_1^\ka] s^{\ga+\ka-\al-\la} \lp \frac{1}{2^{\ga-\al-\la}} \rp^{n+1} \label{relation3}\\
&\leq & c_{\psi,x} \lln t-s \rrn^\ka T_0^{\ga-\ka} \cn[y-\yti;\cac_1^\ka]\lp \frac{1}{2^{\ga-\al-\la}} \rp^{n+1}.\nonumber
\end{eqnarray}
To estimate $\norm{ K }$, remember that
\begin{multline*}  
\norm{ \psi(y_{s_{n+1}^{2i+1}})-\psi(\yti_{s_{n+1}^{2i+1}})-\psi(y_{s_{n+1}^{2i}})+\psi(\yti_{s_{n+1}^{2i}}) } \\
\leq c_\psi \lcl 1+\cn[y;\cac_1^\ka]+\cn[\yti;\cac_1^\ka] \rcl \cn[y-\yti;\cac_1^\ka ] \frac{s^\ka}{(2^{n+1})^\ka},
\end{multline*}
which, together with (\ref{estimation-f-s-t}) applied with $\be=\ga-\al$, gives 

\begin{eqnarray}
\norm{K} & \leq & c_{\psi,x}  \lln t-s \rrn^{\ga-\al}\lcl 1+2A_{T_0} \rcl \cn[y-\yti;\cac_1^\ka]s^\ka \lp \frac{1}{2^{\ka+\ga}} \rp^{n+1} \sum_{i=0}^{2^n-1} \lp 1-\frac{2i}{2^{n+1}} \rp^{-\ga} \nonumber\\
& \leq & c_{\psi,x}  \lln t-s \rrn^\ka \lln t-s\rrn^{\ga-\al-\ka}  \lcl 1+2A_{T_0} \rcl   \nonumber \\
& & \hspace{5cm} \cn[y-\yti;\cac_1^\ka] T_0^\ka\lp \frac{1}{2^{\ka+\ga-1}}\rp^{n+1} \int_0^1 \frac{du}{(1-u)^{\ga}} \label{relation4} \\
&\leq & c_{\psi,x}   \lln t-s \rrn^\ka T_0^{\ga-\al} \lcl 1+2A_{T_0} \rcl\cn[y-\yti;\cac_1^\ka] \lp \frac{1}{2^{\ka+\ga-1}} \rp^{n+1}.\nonumber
\end{eqnarray}

\smallskip

As a result, combining the estimates for $H$ and $K$ along the same lines as for the term $III_{st}$, we end up with:
$$
\norm{IV_{st}}\le
\norm{ J_0} +c_{\psi,x} \lcl 1+2A_{T_0} \rcl \cn[y-\yti;\cac_1^\ka] \lln t-s \rrn^\ka T_0^{\ga-\al}.
$$
But $J_0=\lc t^{-\al}-s^{-\al}\rc \lc \psi(y_0) -\psi(\yti_0) \rc (\der x)_{0s}=0$, so that finally
$$
\norm{IV_{st}}\le
c_{\psi,x} T_0^{\ga-\al}\lcl 1+2A_{T_0} \rcl\cn[y-\yti;\cac_1^\ka] \lln t-s \rrn^\ka.
$$

\smallskip

We have thus proved that $$\cn[z-\zti;\cac_1^\ka] \leq c_{\psi,x} T_0^{\ga-\al-\ka} \lcl 1+2A_{T_0} \rcl \cn[y-\yti;\cac_1^\ka].$$
The contraction property then clearly holds when $\Gamma$ is restricted to a stable ball $\cb_{a,T_0}$, for $T_0$ small enough. This easily yields the existence and uniqueness of a solution to (\ref{eq:volterra-singular}) on $[0,T_0]$.

\end{proof}

\smallskip

The following proposition summarizes the extension of the unique solution to (\ref{eq:volterra-singular}) to an arbitrary interval.
\begin{proposition}[Global existence]
Under the same hypothesis as for Theorem \ref{theo:singular-case}, the local solution $\yun \in \cac_1^\ka([0,T_0])$ can be extended in a unique way into a global solution in $\cac_1^\ka([0,T])$.
\end{proposition}

\begin{proof}
We resort to the same scheme as in Proposition \ref{extension-cas-young}, in which we try to exploit the estimations of the previous proof.

\smallskip

\noindent
\textit{Step 1: Invariance of a ball.} Let $\ep >0$ and $y \in \cac^\ka([0,T_0+\ep])$ such that $y_{|[0,T_0]}=\yun$. Set
$$z_t=\Gamma(y)_t=\begin{cases}
\yun_t & \text{if $t\in [0,T_0]$}\\
a+\int_0^t (t-u)^{-\al} \psi(y_u) \, dx_u & \text{if $t\in [T_0,T_0+\ep]$}.
\end{cases}$$

\smallskip

Let $s,t \in [T_0,T_0+\ep]$ and consider the decomposition (\ref{decompo}) of $(\der z)_{st}$. For $I$, use (\ref{decompojn}), together with the estimations (\ref{termeA}), (\ref{termeB}) and (\ref{j0}), to deduce
\bean
\norm{ \int_s^t (t-u)^{-\al} \psi(y_u) \,dx_u } \leq c_{\psi,x} \lln t-s \rrn^\ka \lcl 1+\ep^{\ga-\al} \cn[y;\cac_1^\ka] \rcl.
\eean
As for $II$, use (\ref{decompojn2}), together with (\ref{termeAbis}), (\ref{terme-E}) and (\ref{j02}) to assert
\bean
\norm{\int_0^s \lc  (t-u)^{-\al}-(s-u)^{-\al}\rc\psi(y_u) \, dx_u } \leq c_{\psi,x}\lln t-s \rrn^\ka \lcl 1+\ep^{\ga-\al-\ka} \cn[y;\cac_1^\ka] \rcl .
\eean

As a result,
$$\cn[z;\cac_1^\ka([T_0,T_0+\ep])] \leq c_{\psi,x} \lcl 1+\ep^{\ga-\al-\ka} \cn[y;\cac_1^\ka] \rcl.$$
By copying the arguments of the proof of Proposition \ref{extension-cas-young}, we then deduce the existence of a small $\ep$, independent of $\yun$, and a radius $N_1$, such that the ball
$$\cb_{\yun,T_0,\ep}:= \lcl y\in \cac_1^\ka([0,T_0+\ep]): \ y_{|[0,T_0]}=\yun, \ \cn[y;\cac_1^\ka] \leq N_1\rcl$$
is invariant by $\Gamma$.

\smallskip

\noindent
\textit{Step 2: Contraction property.} Let $\eta \leq \ep$ , and consider $y,\yti \in \cac_1^\ka([0,T_0+\eta])$ such that $y_{|[0,T_0]}=\yti_{|[0,T_0]}=\yun$, $\cn[y;\cac_1^\ka ] \leq N_1$ and $\cn[\yti;\cac_1^\ka ] \leq N_1$. Set $z=\Gamma(y)$, $\zti=\Gamma(\yti)$.

\smallskip

Let $s,t\in [T_0,T_0+\eta]$ and consider the decomposition (\ref{decompocontra}) of $\der(z-\zti)_{st}$. For $III$, use (\ref{decompojn3}), together with (\ref{relation1}), (\ref{relation2}) and (\ref{j03}), to obtain
$$\norm{ III_{st}}     \leq c_{\psi,x} \eta^{\ga-\al-\ka}\lln t-s\rrn^\ka \lcl 1+2N_1 \rcl \cn[y-\yti;\cac_1^\ka].$$
As far as $IV$ is concerned, the decomposition (\ref{decompo-H-J}), together with (\ref{relation3}), (\ref{relation4}) and the fact that $\psi(y_0)=\psi(\yti_0)$, provides
$$\norm{IV_{st}} \leq c_{\psi,x}\eta^{\la-\ka}  \lln t-s\rrn^\ka \lcl 1+2N_1 \rcl \cn[y-\yti;\cac_1^\ka ].$$

\smallskip

Therefore,
$$\cn[z-\zti;\cac_1^\ka([T_0,T_0+\eta ])] \leq c_{\psi,x} \eta^{\la-\ka} \lcl 1+2N_1 \rcl \cn[y-\yti;\cac_1^\ka ]. $$
The end of the proof follows then exactly the same lines as the proof of Proposition \ref{extension-cas-young}.

\end{proof}


\section{The rough case}
\label{sec:rough-case}
In this section, we go back to equation (\ref{eq:volterra}), with a smooth and bounded coefficient $\si$. However, we will only assume that $x$ belongs to $\cac_1^\ga([0,T]; \R^n)$ for some $\ga \in (1/3,1/2)$, which means in particular that we can no longer resort to Young's interpretation for $\int_0^t \si(t,u,y_u) \, dx_u$ and some rough path type considerations must come into the picture. We will thus briefly review the setting used in this context, and then prove a local existence and uniqueness result for our equation.

\subsection{Controlled processes}
For sake of conciseness, we only recall here the key ingredients of the formalism introduced in \cite{Gu} in order to handle integrals driven by an irregular signal $x$. First, as usual in the rough path theory, we will have to assume a priori the following hypothesis:

\begin{hypothesis}\label{hyp:x2}
The path $x$ admits a Levy area, that is a process $x^2 \in \cac_2^{2\ga}([0,T];\R^{n,n})$ such that
$$\der x^2=\der x \otimes \der x , \quad \mbox{i.e.} \quad (\der x^2)_{sut}(i,j)=(\der x^i)_{su} \otimes (\der x^j)_{ut}, 
$$
for all $s,u,t \in [0,T]$ and $i,j \in \{ 1,\cdots,n\}$.
\end{hypothesis}

As explained in \cite{Gu}, we are then incited to introduce a particular subspace of the space of Hölder continuous functions $\cac_1^\ga([0,T];\R^{1,k})$, which are the convenient processes to be integrated with respect to $x$:
\begin{definition}
Let $k\in \N^\ast$ and $\eta > \ga$. A process $y \in \cac_1^\ga([0,T];\R^{1,k})$ is said to be $(\ga,\eta)$-controlled by $x$ if there exists $y' \in \cac_1^{\eta-\ga}([0,T];\cl(\R^n,\R^{1,k}))$, $r^y\in \cac_2^{\eta}([0,T];\R^{1,k})$ such that
\begin{equation}\label{decompo-controlled-path}
(\der y)_{st}=y'_s (\der x)_{st}+r^y_{st}, \quad \mbox{for any} \ s,t \in [0,T].
\end{equation}
\end{definition}

\begin{remark}
The decomposition (\ref{decompo-controlled-path}) is not necessarily unique. However, if we fix $y,y'$, then, of course, the remainder $r^y$ is uniquely determined. For this reason, we shall denote $\cq^{\ga,\eta}([0,T];\R^{1,k})$ the space of couples $(y,y') \in \cac_1^\ga([0,;\R^{1,k}) \times \cac_1^{\eta-\ga}([0,T];\cl(\R^n,\R^{1,k}))$ such that the decomposition (\ref{decompo-controlled-path}) holds. This space is endowed with the natural semi-norm
\begin{align*}
&\cn[y;\cq^{\ga,\eta}([0,T];\R^{1,k})]=\cn[(y,y');\cq^{\ga,\eta}([0,T];\R^{1,k})]\\
&:=\cn[y;\cac_1^\ga([0,T];\R^{1,k})]+\cn[y';\cac_1^0([0,T];\cl(\R^n,\R^{1,k})]
+\cn[y';\cac_1^{\ga-\eta}([0,T];\cl(\R^n,\R^{1,k})]  \\
&\hspace{11cm}+\cn[r^y;\cac_2^{\eta}([0,T];\R^{1,k})].
\end{align*}
\end{remark}

Observe that if $(y,y')\in \cq^{\ga,\eta}([0,T];\R^{1,k})$, then 
\begin{equation}\label{lien-normes}
\cn[y;\cac_1^\ga([0,T];\R^{1,d})] \leq c_x \lcl \norm{y'_0}+T^{\eta-\ga} \cn[y;\cq^{\ga,\eta}([0,T];\R^{1,d})] \rcl.
\end{equation}
Finally, let us denote $\cq^\ga([0,T];\R^{1,k})=\cq^{\ga,2\ga}([0,T];\R^{1,k})$.

\smallskip

With our main equation (\ref{eq:volterra-young}) in mind, it is important for us to get a stability property for controlled processes, when composed with the map $\si$. This is the object of the following proposition (for which we recall the notation on gradient of functions given at the end of the introduction).

\begin{proposition}\label{stabilite-sigma}
Let $(y,y')\in \cq^\ga([0,T];\R^{1,d})$, with decomposition $\der y=y' (\der x)+r^y$, and consider $\si \in \cac^{2,\textbf{\textit{b}}}([0,T]^2 \times \R^{1,d};\R^{d,n})$. For $i=1,\cdots, d$, denote by $\si_i(z)$ the $i\textsuperscript{th}$ line of $\si(z)$ when considered as a matrix. Then, for any $t\geq 0$, $(\si_i(t,.,y_.),D_3\si_i(t,.,y_.) \circ y') \in \cq^\ga([0,T]; \R^{1,n})$ and
\begin{equation}\label{eq:stabilite-sigma}
\cn[\si_i(t,.,y_.); \cq^\ga([0,T];\R^{1,n})]\leq c_\si \lcl 1+\cn[y;\cq^\ga([0,T];\R^{1,d})]^2\rcl,
\end{equation}
where $c_\si$ does not depend on $t$.
\end{proposition}
\begin{proof}
See Appendix.
\end{proof}

Let us now turn to the integration of weakly controlled paths, which is summarized in the following proposition, borrowed from \cite{Gu}. This result requires a little additional notation: if $\varphi \in \cl(\R^n, \R^{1,n})$ and $A\in \R^{n,n}$, we denote $\varphi \cdot A=\sum_{i,j=1}^n \left\langle  \varphi e_i,e_j^\ast \right\rangle A_{ij}$.

\begin{proposition}\label{prop:integration-controlled-paths}
Let $x$ be a signal satisfying Hypothesis \ref{hyp:x2}, and let also $(z,z')$ be an element of $\cq^\ga([0,T];\R^{1,n})$ with decomposition $\der z =z'(\der x)+r^z$. One can define $A\in \cac_1^\ga([0,T];\R)$ by $A_0=a \in \R$ and 
$$(\der A)_{st}=z_s(\der x)_{st}+z'_s \cdot x^2_{st}+\laa_{st}(r^z \der x+\der z' \cdot x^2),$$
and set $\cj(z \, dx)=\cj((z,z') \, dx)=\der A$. Then $\cj(z \, dx)$ coincides with the usual Rieman integral of $z$ with respect to $x$ in case of smooth functions. Moreover, it holds
$$\cj(z \, dx)=\lim_{|\Pi_{st}|\to 0} \sum_i \lcl z_{t_i}(\der x)_{t_it_{i+1}}+z'_{t_i}\cdot x^2_{t_it_{i+1}}\rcl,$$
for any $0\leq s <t\leq T$, where the limit is taken over all the partitions $\Pi_{st}=\{s =t_0 <t_1< \ldots < t_n=t\}$ of $[s,t]$, as the mesh of the partition goes to zero.
\end{proposition}

It only remains to enunciate the multidimensional version of the previous proposition:

\begin{definition}\label{cor:integration-controlled-paths}
Assume that $z\in \cac_1^\ga([0,T];\R^{d,n})$ is such that for each $z_i$ ($i\textsuperscript{th}$ line of $z$), there exists $z'_i\in \cac_1^\ga ([0,T];\cl(\R^n,\R^{1,n}))$ for which $(z_i,z'_i) \in \cq^\ga([0,T];\R^{1,n})$. Then we define $\cj(z \, dx)=\cj((z,z') \, dx) \in \cac_1^\ga([0,T];\R^{1,d})$ by the natural relations
$$\cj(z \, dx)^{(i)}=\cj((z_i,z'_i) \, dx), \quad i=1,\ldots, d.$$
\end{definition}

\subsection{Rough Volterra equations}
Let us say a few words about the strategy to be used in order to solve equation (\ref{eq:volterra-young}) in case of a rough driving signal. First, this Volterra system will be interpreted according to Propositions \ref{stabilite-sigma} and \ref{prop:integration-controlled-paths} when $(y,y')$ belongs to $\cq^\ga([0,T];\R^{1,d})$ and $\si \in \cac^{2,\textbf{\textit{b}}}([0,T]^2 \times \R^{1,d};\R^{d,n})$. Moreover, in order to settle a fixed point argument, we shall see that the process $z$ defined by $z_0=a$ and
$$(\der z)_{st}=\cj_{st}(\si(t,.,y_.) \, dx)+\cj_{0s}([\si^t-\si^s](\cy) \, dx)$$
is a controlled process (recall that $\cy$ stands for the multidimensional function $s\mapsto(s,y_s)$). Indeed, if we assume that the path $w_i=\si_i^t(\cy)$ can be decomposed as 
$$
\der w_i=\der \si_i^t(\cy)=\si_i^t(\cy)'(\der x)+r^{\si_i^t(\cy)},
$$ 
which can be done owing to Proposition \ref{stabilite-sigma}, and if we set $\der z^{(i)}=\cj(w_i\, dx)$, then one can write $(\der z)_{st}^{(i)}=\si_i(s,s,y_s)(\der x)_{st}+(r^z_{st})^{(i)}$ for $i=1,\ldots,d$, with
\begin{multline*}
(r^z_{st})^{(i)}=\lc \si_i(t,s,y_s)-\si_i(s,s,y_s)\rc (\der x)_{st}+\si_i^t(\cy)'_s \cdot x^2_{st}+\laa_{st}(r^{\si_i^t(\cy)} \der x+\der (\si_i^t(\cy)') \cdot x^2)\\
+\cj_{0s}(\lc \si(t,.,y_.)-\si(s,.,y_.)\rc \, dx)^{(i)}.
\end{multline*}
If we manage to show that $\si(.,.,y_.)^\ast: x \mapsto (\si_1(.,.,y_.)(x),\ldots, (\si_d(.,.,y_.)(x))$ belongs to $\cac_1^\ga([0,T];\cl(\R^n,\R^{1,d}))$ and $r^z \in \cac_2^{2\ga}([0,T];\R^{1,d})$ (which will be done in the course of the following proof), then $(z, \si(.,.,y_.)^\ast) \in \cq^\ga([0,T];\R^{1,d})$ and the application $\Gamma$ introduced in the Young setting becomes here
\begin{equation}\label{gamma-rough}
\Gamma: \cq^\ga([0,T];\R^{1,d}) \to \cq^\ga([0,T];\R^{1,d}), \ (y,y') \mapsto (z, \si(.,.,y_.)^\ast).
\end{equation}
With this notation, a solution of (\ref{eq:volterra-young}) corresponds to a fixed point of $\Gamma$.

\smallskip

We have now all the tools in hand to express the announced (local) result properly:

\begin{theorem}\label{theo-cas-rough}
Let $\ka \in (0,1)$ such that $\ga(\ka +2)>1$, $\si \in \cac^{3,\textbf{\textit{b}},\ka}([0,T]^2 \times \R^d; \R^{d,n})$ and $a\in \R^{1,d}$. Then there exists $T_0 \in (0,T]$ such that the equation 
$$y_t=a+\cj_{0t}(\si(t,.,y_.) \, dx),$$
interpreted in the sense of Definition \ref{cor:integration-controlled-paths}, admits a unique solution in $\cq^\ga([0,T_0]; \R^{1,d})$.
\end{theorem} 

As in the Young case, the result will stem from a contraction argument (Proposition \ref{contraction-property-rough}) on some invariant ball (Proposition \ref{boule-invariante-rough}). Before we turn to detail these arguments, let us state an equivalent of Lemma \ref{estimation-sigma}:
\begin{lemma}\label{estimation-sigma-rough}
Let $(y,y'),(\yti,\yti') \in \cq^\ga([0,T];\R^{1,d})$ such that $y_0=\yti_0$ and $y'_0=\yti'_0$. Then, under the hypothesis of Theorem \ref{theo-cas-rough}, for any $s,t \in [0,T]$,
\begin{equation}\label{estimation-terme-retard-1}
\cn[[\si_i^t-\si_i^s](\cy);\cq^\ga([0,T];\R^{1,n})] \leq c_\si \lln t-s \rrn \lcl 1+\cn[y;\cq^\ga([0,T];\R^{1,d})]^2 \rcl,
\end{equation}
the path $\si^t(\cy)-\si^t(\cyti)$ satisfies
\begin{align}\label{estimation-terme-retard-2}
&\cn[\si_i^t(\cy)-\si_i^t(\cyti); \cq^\ga([0,T];\R^{1,d})] \\
&\leq c_\si \lcl 1+\cn[y;\cq^\ga([0,T];\R^{1,d})]^2+\cn[\yti;\cq^\ga([0,T];\R^{1,d})]^2 \rcl \cn[y-\yti;\cq^\ga([0,T];\R^{1,d})], \nonumber
\end{align}
and
\begin{align}\label{estimation-terme-retard-3}
&\cn[[ \si_i^t-\si_i^s ](\cy)-[ \si_i^t-\si_i^s ](\cyti); \cq^{\ga,\ga+\ga\ka}([0,T];\R^{1,d})] 
\leq c_\si \lln t-s \rrn  \\
&\times\lcl 1+\cn[y;\cq^\ga([0,T];\R^{1,d})]^{1+\ka}+\cn[\yti;\cq^\ga([0,T];\R^{1,d})]^{1+\ka} \rcl 
 \cn[y-\yti;\cq^\ga([0,T];\R^{1,d})]. \nonumber
\end{align}
\end{lemma}
\begin{proof}
See Appendix.

\end{proof}

We can now state the result concerning the invariance of a ball for the map $\Gamma$:
\begin{proposition}[Invariance of a ball]\label{boule-invariante-rough}
Under the hypothesis of Theorem \ref{theo-cas-rough}, there exists $T_0 \in (0,T]$ such that for each $T_1 \in (0,T_0]$, the ball
$$B_{T_1}^{A_{T_1}}=\{ (y,y') \in \cq^\ga([0,T_1]): \ y_0=a, \ y'_0=\si(0,0,a)^\ast, \ \cn[(y,y');\cq^\ga([0,T_1])] \leq A_{T_1} \}$$
is invariant by $\Gamma$ (defined by (\ref{gamma-rough})) for some large enough radius $A_{T_1}$.
\end{proposition}

\begin{proof}
Fix a time $T_0 \leq T$ and let $(y,y')\in B_{T_0}^{A_{T_0}}$ with decomposition $\der y=y'\der x +r^y$. Set $(z,z')=\Gamma(y,y')$. Then $\der z=z' \der x+r^z$, where $r^z$ can be further decomposed into:
\begin{equation}\label{decomposition-reste}
r^z=r^{z,0}+r^{z,1,1}+r^{z,1,2}+r^{z,2,1}+r^{z,2,2},
\end{equation}
with
\begin{eqnarray*}
r^{z,0,(i)}_{st}&=&\lc \si^t_i-\si^s_i \rc (\cy_s)(\der x)_{st}, \qquad  
r^{z,1,1,(i)}_{st}=\si_i^t(\cy)'_s \cdot x^2_{st}  \\
r^{z,1,2,(i)}_{st}&=&\laa_{st}\big(r^{\si_i^t(\cy)} \der x+\der (\si_i^t(\cy)') \cdot x^2 \big),
\end{eqnarray*}
and 
\bean
r^{z,2,1,(i)}_{st}&=&\lc \si_i^t-\si_i^s \rc (\cy_0) (\der x )_{0s}+  \lc \si^t_i-\si^s_i \rc (\cy) '_0  \cdot x^2_{0s}  \\
r^{z,2,2,(i)}_{st}&=&\laa_{0s} \lp \lc r^{\si_i^t(\cy)}-r^{\si_i^s(\cy)} \rc \der x+\der(\lc \si_i^t-\si_i^s \rc (\cy)') \cdot x^2 \rp.
\eean

Let us check that this decomposition actually identifies $z$ as an element of $\cq^\ga$, that is $z' \in \cac_1^\ga$ and $r^z \in \cac_2^{2\ga}$. For $z'$, pick $0\leq s <t \leq T_1$ and observe that
\bean
\norm{(\der z')_{st}} & =& \norm{\si(t,t,y_t)^\ast-\si(s,s,y_s)^\ast }\\
&\leq & \norm{ \si(t,t,y_t)^\ast-\si(s,t,y_t)^\ast }+\norm{\si(s,t,y_t)^\ast-\si(s,s,y_s)^\ast }\\
&\leq & \norm{D\si }_\infty \lln t-s \rrn+\sum_{i=1}^d \norm{ \der( \si_i^s(\cy))_{st} }.
\eean
But, according to (\ref{lien-normes}),
\begin{multline*}
\norm{ \der(\si_i^s(\cy))_{st} }  \leq  c_x \lln t-s \rrn^\ga \lcl \norm{D_3\si_i(s,\cy_0) \circ y'_0 }+T_0^\ga \cn[\si_i^s(\cy); \cq^\ga]\rcl \\
\leq c_{x,\si} \lln t-s \rrn^\ga \lcl 1+T_0^\ga \cn[\si_i^s(\cy); \cq^\ga]\rcl,
\end{multline*}
which, together with (\ref{eq:stabilite-sigma}), leads to $\cn[z'; \cac_1^\ga] \leq c_{x,\si} \lcl 1+T_0^\ga \cn[y;\cq^\ga]^2 \rcl$.

\smallskip

Let us now estimate the $2\ga$-Hölder norm of the remaining terms.

\smallskip

\noindent
\textit{Case of $r^{z,0}$:} Clearly, $\cn[r^{z,0};\cac_2^{2\ga}] \leq \norm{D\si}_\infty \cn[x;\cac_1^\ga] T_0^{1-\ga} \leq c_{\si,x}$.

\smallskip

\noindent
\textit{Case of $r^{z,1,1}$:} Since $\norm{\si_i^t(\cy)'_0}=\norm{D_3\si_i(t,\cy_0) \circ y_0' } \leq c_\si$, one has, owing to (\ref{eq:stabilite-sigma}),
\bean
\norm{r^{z,1,1,(i)}_{st} } &\leq & c_\si \lln t-s \rrn^{2\ga} \cn[x^2;\cac_2^{2\ga}] \lcl 1+T_0^\ga \cn[\si_i^t(\cy)';\cac_1^\ga] \rcl\\
&\leq & c_{\si,x} \lln t-s \rrn^{2\ga} \lcl 1+T_0^\ga \cn[\si_i^t(\cy);\cq^\ga]\rcl \ \leq \ c_{\si,x} \lln t-s \rrn^{2\ga} \lcl 1+T_0^\ga \cn[y;\cq^\ga]^2 \rcl . 
\eean

\smallskip

\noindent
\textit{Case of $r^{z,1,2}$:} It is readily checked, invoking (\ref{contraction}) and (\ref{eq:stabilite-sigma}), that
\bean
\norm{r_{st}^{z,1,2,(i)}} &\leq & c \lln t-s \rrn^{3\ga}  \lcl \cn[r^{\si_i^t(\cy)};\cac_2^{2\ga}] \cn[x;\cac_1^\ga]+\cn[(\si_i^t(\cy))';\cac_1^\ga] \cn[x^2;\cac_2^{2\ga}] \rcl\\
&\leq & c_x \lln t-s \rrn^{3\ga}  \cn[\si_i(t,\cy);\cq^\ga ] \ \leq \ c_{x,\si} \lln t-s \rrn^{2\ga} T_0^\ga \lcl 1+\cn[y;\cq^\ga]^2 \rcl.
\eean

\smallskip

\noindent
\textit{Case of $r^{z,2,1}$:} The following elementary estimates hold true.
\begin{eqnarray*}
\norm{r^{z,2,1,(i)}_{st} } &\leq& \norm{D\si_i }_\infty \lln t-s \rrn T_0^\ga \cn[x;\cac_1^\ga]+ \norm{ D_3\si_i(t,\cy_0)-D_3\si_i(s,\cy_0)} \norm{y'_0 } \cn[x^2;\cac_2^{2\ga}] T_0^{2\ga}\\
&\leq& c_{x,\si} \lln t-s \rrn^{2\ga}.
\end{eqnarray*}

\smallskip

\noindent
\textit{Case of $r^{z,2,2}$:} Owing to (\ref{contraction}) and (\ref{estimation-terme-retard-1}), we have
\bean
\lefteqn{\norm{r^{z,2,2,(i)}_{st}}}\\
 & \leq & c \, T_0^{3\ga} \lcl \cn[r^{\si_i^t(\cy)}-r^{\si_i^s(\cy)};\cac_2^{2\ga}] \cn[x;\cac_1^\ga]+\cn[ (\lc \si_i^t-\si_i^s\rc (\cy))';\cac_1^\ga] \cn[x^2;\cac_2^{2\ga}]\rcl\\
&\leq & c_x \, T_0^{3\ga}  \cn[[\si_i^t-\si_i^s](\cy);\cq^\ga] \ \leq \ c_{x,\si} \, T_0^{3\ga} \lln t-s \rrn \lcl 1+\cn[y;\cq^\ga]^2 \rcl .
\eean

\smallskip

Finally, gathering all our estimates for the terms in (\ref{decomposition-reste}), it is easily seen that $\cn[r^z;\cac_2^{2\ga}]$  $\leq c_{\si,x} \lcl 1+T_0^\ga \cn[y;\cq^\ga]^2 \rcl$. Hence we have obtained that $r^z \in \cac_2^{2\ga}$ and $(z,z')\in \cq^\ga$. 

\smallskip

Notice that the above estimations also easily lead to $\cn[z;\cq^\ga] \leq c_{x,\si} \lcl 1+T_0^\ga \cn[y;\cq^\ga]^2 \rcl$. Choose now for $T_0$ the greatest time $\tau \in (0,T]$ such that the equation $c_{\si,x}\lcl 1+\tau^\ga A \rcl =A$ admits a unique solution $A_{\tau}$.  Then $T_0$ satisfies the property announced in our proposition.

\end{proof}

\smallskip

We can now prove the contraction property allowing to establish the existence and uniqueness of a local solution to equation (\ref{eq:volterra-young}).
\begin{proposition}[Contraction property]\label{contraction-property-rough}
Under the hypothesis of Theorem \ref{theo-cas-rough}, there exists $T_1 \in (0,T_0]$ such that for each $T_2 <T_1$, the application $\Gamma$ is a strict contraction on the (stable) ball $B_{T_2}^{A_{T_2}}$.
\end{proposition}
\begin{proof}
Let $(y,y'),(\yti,\yti')$ two elements of $B_{T_1}^{A_{T_1}}$, and set $(z,z')=\Gamma(y,y')$, $(\zti,\zti')=\Gamma(\yti,\yti')$. Thus, $\der(z-\zti)=(z'-\zti')\der x+(r^z-r^{\zti})$, where $z'=\si(.,.,y_.)^\ast$, $\zti'=\si(.,.,\yti_.)^\ast$, and $r^z$ is given by (\ref{decomposition-reste}), with a similar expression for $r^{\zti}$. Let us now estimate each term of 
$$\cn[z-\zti;\cq^\ga]=\cn[z'-\zti';\cac_1^0]+\cn[z'-\zti';\cac_1^\ga]+\cn[r^z-r^{\zti};\cac_2^{2\ga}]+\cn[z-\zti;\cac_1^\ga].$$

\smallskip

\noindent
\textit{Case of $\cn[z'-\zti';\cac_1^0]$:} If $s\in [0,T_1]$, $\norm{z'_s-\zti'_s }=\norm{\si(s,s,y_s)^\ast-\si(s,s,\yti_s)^\ast } \leq \norm{D\si }_\infty \norm{y_s-\yti_s}$.
But $y_0=\yti_0$, so that $\norm{y_s-\yti_s } \leq T_1^\ga \cn[y-\yti;\cac_1^\ga]$ and $\cn[z'-\zti';\cac_1^0] \leq c_\si T_1^\ga \cn[y-\yti;\cq^\ga]$.

\smallskip

\noindent
\textit{Case of $\cn[z'-\zti';\cac_1^\ga]$:} Pick $0\leq s <t \leq T_1$ and observe that
\bean
\norm{(z'_t-\zti'_t)-(z'_s-\zti'_s)} &=& \norm{(\si(t,\cy_t)^\ast-\si(t,\cyti_t)^\ast-\si(s,\cy_s)^\ast+\si(s,\cyti_s)^\ast }\\
&\leq & \norm{[\si^t-\si^s](\cy_t)-[\si^t-\si^s](\cy_t)}+\norm{ \der(  \si^s(\cy)-\si^s(\cyti))_{st}}.
\eean
Then
\bean
\norm{[\si^t-\si^s](\cy_t)-[\si^t-\si^s](\cy_t)} & \leq & \norm{D(\si^t-\si^s)}_\infty \norm{y_t-\yti_t}\\
&\leq & \norm{D^2\si}_\infty \lln t-s \rrn \cn[y-\yti;\cac_1^\ga] T_1^\ga \\
& \leq & c_\si \lln t-s \rrn^\ga \cn[y-\yti;\cq^\ga] T_1,
\eean
while, according to (\ref{lien-normes}) and (\ref{estimation-terme-retard-2}),
\bean
\norm{ \der(  \si_i^s(\cy)-\si_i^s(\cyti))_{st}}
 &\leq & \lln t-s\rrn^\ga \cn[\si_i^s(\cy)-\si_i^s(\cyti);\cac_1^\ga] \\
 &\leq & c_{x} \lln t-s \rrn^\ga \lcl \norm{(\si_i^s(\cy)-\si_i^s(\cyti))'_0 }+T_1^\ga \cn[\si_i^s(\cy)-\si_i^s(\cyti);\cq^\ga]\rcl\\
 &\leq & c_{x,\si} \lln t-s \rrn^\ga T_1^\ga \lcl 1+\cn[y;\cq^\ga]^2+\cn[\yti;\cq^\ga]^2 \rcl \cn[y-\yti;\cq^\ga]
\eean
since $(\si_i^s(\cy)-\si_i^s(\cyti))'_0=0$. Hence, thanks to the fact that we are working on the invariant ball $B_{T_1}^{A_{T_1}}$, we get $\cn[z'-\zti';\cac_1^\ga] \leq c_{x,\si} \lcl 1+A_{T_1}^2 \rcl \cn[y-\yti;\cq^\ga] T_1^\ga$.

\smallskip

\noindent
\textit{Case of $\cn[r^z-r^{\zti};\cac_2^{2\ga}]$:} Since $(y_0,y'_0)=(\yti_0,\yti'_0)$, $r^{z-\zti}=r^z-r^{\zti}$ reduces to the sum of
\begin{align*}
&r_{st}^{z-\zti,0,(i)}=\{ [\si_i^t-\si_i^s](\cy_s)-[\si_i^t-\si_i^s](\cyti_s) \} (\der x)_{st},
\quad r_{st}^{z-\zti,1,1,(i)}=[\si_i^t(\cy)'_s-\si_i^t(\cyti)'_s] \cdot x^2_{st}  \\
&r_{st}^{z-\zti,1,2,(i)}=\laa_{st}([r^{\si_i^t(\cy)}-r^{\si_i^t(\cyti)} ] \der x+\der(\si_i^t(\cy)'-\si_i^t(\cyti)') \cdot x^2)  \\
&r_{st}^{z-\zti,2,(i)}=\laa_{0s}([r^{\si_i^t(\cy)}-r^{\si_i^s(\cy)}-r^{\si_i^t(\cyti)}+r^{\si_i^s(\cyti)}] \der x +\der([\si_i^t-\si_i^s](\cy)'-[\si_i^t-\si_i^s](\cyti)' \cdot x^2).
\end{align*}
We will now bound each of these terms.

\smallskip

\noindent
\textit{Study of $r^{z-\zti,0}_{st}$:} One has
\bean
\norm{r^{z-\zti,0,(i)}_{st}} &\leq & c_x \lln t-s\rrn^\ga \norm{D(\si_i^t-\si_i^s)}_\infty \norm{ \cy_s-\cyti_s}\\
&\leq & c_x \lln t-s \rrn^{1+\ga} \norm{D^2\si_i }_\infty \norm{y_s-\yti_s}\\
&\leq & c_{x,\si} \lln t-s \rrn^{2\ga} \cn[y-\yti;\cac_1^\ga] T_1^{1-\ga}\ \leq \ c_{x,\si} \lln t-s \rrn^{2\ga} \cn[y-\yti;\cq^\ga] T_1^{1-\ga}.
\eean

\smallskip

\noindent
\textit{Study of $r^{z-\zti,1,1}_{st}$:} Since $(\si^t(\cy)-\si^t(\cyti))'_0=0$, we get, owing to (\ref{estimation-terme-retard-2}),
\bean
\norm{r^{z-\zti,1,1,(i)}_{st}} &\leq & c_x \lln t-s \rrn^{2\ga} \norm{(\si_i^t(\cy)-\si_i^t(\cyti))'_s } \ \leq \ c_x \lln t-s \rrn^{2\ga} \cn[\si_i^t(\cy)-\si_i^t(\cyti);\cq^\ga] T_1^\ga\\
&\leq & c_x \lln t-s \rrn^{2\ga} \lcl 1+\cn[y;\cq^\ga]^2+\cn[\yti;\cq^\ga]^2 \rcl \cn[y-\yti;\cq^\ga] T_1^\ga .
\eean

\smallskip

\noindent
\textit{Study of $r^{z-\zti,1,2}$:} By (\ref{contraction}) and (\ref{estimation-terme-retard-2}),
\bean
\norm{r^{z-\zti,1,2,(i)}_{st}} & \leq & c_x \lln t-s \rrn^{3\ga} \cn[\si_i^t(\cy)-\si_i^t(\cyti);\cq^\ga]\\
&\leq & c_{\si,x} \lln t-s \rrn^{2\ga} \lcl 1+\cn[y;\cq^\ga]^2+\cn[\yti;\cq^\ga]^2 \rcl \cn[y-\yti;\cq^\ga] T_1^\ga.
\eean

\smallskip

\noindent
\textit{Study of $r^{z-\zti,2}$:} By (\ref{contraction}) and (\ref{estimation-terme-retard-3}),
\bean
\norm{r^{z-\zti,2,(i)}_{st}} &\leq & c_x T_1^{\ga(\ka+2)}\cn[[\si_i^t-\si_i^s](\cy)-[\si_i^t-\si_i^s](\cyti);\cq^{\ga,\ga(1+\ka)}]\\
&\leq & c_{x,\si} T_1^{\ga(\ka+2)} \lln t-s \rrn  \lcl 1+\cn[y;\cq^\ga]^{1+\ka}+\cn[\yti;\cq^\ga]^{1+\ka} \rcl \cn[y-\yti;\cq^\ga].
\eean
Finally, putting together all our estimates of the remainder terms, we end up with the relation $\cn[r^z-r^{\zti};\cac_2^{2\ga}] \leq c_{x,\si} \lcl 1+A_{T_1}^2 \rcl \cn[y-\yti;\cq^\ga] T_1^\ga$, which together with the above estimation of $\cn[z'-\zti';\cac_1^\ga]$, gives
$$\cn[z-\zti;\cq^\ga] \leq c_{x,\si} \lcl 1+A_{T_1}^2 \rcl \cn[y-\yti;\cq^\ga ] T_1^\ga.$$
The greatest time $T_1 \in (0,T_0]$ such that $c_{x,\si} \lcl 1+A_{T_1}^2 \rcl T_1^\ga \leq 1/2$ then clearly yields the contraction property for $\Gamma$ on $[0,T_1]$.

\end{proof}

\smallskip

In the rough case, it is also easily seen that our existence and uniqueness result for equation (\ref{eq:volterra-young}) can be applied to the fractional Brownian motion:
\begin{corollary}
Let $B$ be a $n$-dimensional fractional Brownian motion with Hurst parameter $1/3<H\le1/2$, defined on a complete probability space $(\Omega,\cf,P)$. Then almost surely, $B$ fulfills the hypotheses of Theorem \ref{theo-cas-rough}.
\end{corollary}

\begin{proof}
We only have to show that $B$ satisfies Hypothesis \ref{hyp:x2}. But this kind of result is easily deduced from the convergence results contained in \cite{CQ}.

\end{proof}

\subsection{Extending the solution}
\label{sec:5.3}

To finish with, let us briefly evoke the technical difficulties we encounter when trying to extend the solution on $[0,T]$ along the same lines as in the Young case. Denote $(y^{(1)},(\yun)')$ the solution on $[0,T_0]$.

\smallskip

The first step would consist in finding some small $\ep >0$, independent of $(\yun,(\yun)')$, and some radius $N_1$ such that the ball
$$\{(y,y') \in \cq^\ga([0,T_0+\ep]): \ (y,y')_{|[0,T_0]}=(\yun,(\yun)'), \ \cn[(y,y');\cq^\ga([0,T_0+\ep])] \leq N_1 \}$$
is invariant by $\Gamma$. In fact, if we set $(z,z')=\Gamma(y,y')$ for $(y,y')$ in this ball, then some standard estimations, similar to those appearing in the proofs above, show that
\begin{equation}\label{tentative-prolongement}
\cn[(z,z');\cq^\ga([0,T_0+\ep])] \leq c_1 \cn[\yun;\cq^\ga([0,T_0])]+c_2 \lcl 1+\ep^\la \cn[(y,y');\cq^\ga([0,T_0+\ep])]^2 \rcl,
\end{equation}
for some $\la >0$ and some constants $c_1,c_2$ with $c_1 >2$. It is then rather clear that, owing to the exponent $2$ in the latter expression, the constant $\ep$ ensuring the stability of the ball has to depend on $\cn[\yun;\cq^\ga([0,T_0])]$.

\smallskip

More specifically, imagine the reasoning of the proof of Proposition \ref{extension-cas-young} remains true when starting with (\ref{tentative-prolongement}), which means that we can find some constant $\ep >0$ and some sequence of radii $(N_i)$ such that 
\beq\label{eq:52}
c_1N_i+c_2 \lcl 1+\ep^\la N_{i+1}^2 \rcl \leq N_{i+1}.
\eeq 
Then $N_{i+1} \geq c_1 N_i \geq 2N_i$ and the sequence $(N_i)$ diverges to infinity. On the other hand, if relation  (\ref{eq:52}) is meant to admit solutions, then the relation $1-4\ep^\la c_2(c_1N_i+c_2) \geq 0$ must be fulfilled, so that $(N_i)$ is bounded, hence a contradiction.

\smallskip

At this point, it is interesting to notice that even if $\ep$ is allowed to vary and becomes a sequence $\ep_i$ such that $\sum_i \ep_i =\infty$ (in order to be sure that $[0,T]$ is covered), then we get $\frac{N_1}{2} 2^i \leq N_i \leq \frac{c}{\ep_{i+1}^\la}$, so that $\ep_i \leq \frac{c}{(2^{1/\la})^i}$, which of course contradicts $\sum_i \ep_i =\infty$. 

\

This failure in our apprehension of (\ref{eq:volterra}) motivated the study of a particular case of Volterra equations (see our companion paper \cite{DT}) for which some modifications of the $\der$-formalism enable to get rid (in some way) of the past-dependent term in (\ref{decomposition-generale}).


\section{Appendix}
We gather in this section some regularity results for the functions and controlled processes we handle in throughout the paper.

\begin{proof}[Proof of Lemma \ref{estimation-sigma}]

To obtain (\ref{inegalite-sigma-1}), pick $u <v$ and observe that
\bean
\norm{ [\si^t-\si^s](\cy_v)-[\si^t-\si^s](\cy_u)} &\leq & \norm{ D(\si^t-\si^s)}_\infty \, \norm{\cy_v-\cy_u }\\
&\leq & \norm{D^2\si}_\infty \, \lln t-s \rrn \lp \lln v-u \rrn +\cn[y; \cac_1^\ga] \, \lln v-u\rrn^\ga \rp,
\eean
which gives the result.

\smallskip

In order to establish (\ref{inegalite-sigma-2}), let us introduce the operator $R$ defined for any $\varphi \in \cac^{1,b}( \R^{d+1})$, $\xi,\xi' \in \R^{d+1}$, by 
$$R\varphi (\xi,\xi')=\int_0^1 D\varphi(\al\xi+(1-\al)\xi') \, d\al.$$
Then of course $\norm{R\varphi}_\infty \leq \norm{D\varphi}_\infty$ and $\norm{ R\varphi(\xi_1,\xi_1')-R\varphi(\xi_2,\xi_2')}\leq \norm{D^2\varphi}_\infty (\norm{\xi_1-\xi_2}+\norm{\xi_1'-\xi_2'})$. With this notation, if $0 <u<v<T$, 
\bean
\lefteqn{\norm{[\si^t(\cy_v)-\si^t(\cyti_v)]-[\si^t(\cy_u)-\si^t(\cyti_u)]}}\\
&=& \norm{R\si^t(\cy_v,\cyti_v)(\cy_v-\cyti_v)-R\si^t(\cy_u,\cyti_u)(\cy_u-\cyti_u)}\\
& \leq & \norm{ R\si^t(\cy_v,\cyti_v)([\cy_v-\cyti_v]-[\cy_u-\cyti_u])}+\norm{ [R\si^t(\cy_v,\cyti_v)-R\si^t(\cy_u,\cyti_u)](\cy_u-\cyti_u)}\\
&\leq & \norm{D\si^t }_\infty\norm{[y_v-\yti_v]-[y_u-\yti_u]}\\
& & \hspace{3cm} +\norm{ D^2 \si^t }_\infty (2\lln v-u\rrn +\norm{y_v-y_u }+\norm{\yti_v-\yti_u }) \norm{y_u-\yti_u }\\
&\leq &  \cn[y-\yti; \cac_1^\ga] \lln v-u \rrn^\ga \lcl \norm{D\si}_\infty +\norm{D^2\si}_\infty  (2T^{1-\ga}+\cn[y;\cac_1^\ga]+\cn[\yti;\cac_1^\ga]) T^\ga \rcl,
\eean
where, in the last inegality, we have used the fact that $y_u-\yti_u=[y_u-\yti_u]-[y_0-\yti_0]$. Inequality (\ref{inegalite-sigma-2}) follows easily. Notice that those are the same arguments as in the proof of \cite[Lemma 5]{Gu}.

\smallskip

To prove (\ref{inegalite-sigma-3}), let us introduce the operator $L$ defined for any $\varphi \in \cac^{2,\textbf{\textit{b}},\ka}(\R^{d+2})$ and any $s,t \in \R$, $\xi,\xi' \in \R^{d+1}$, as
$$L\varphi(s,t,\xi,\xi')=\int_0^1 \int_0^1 D^2\varphi(s+\mu (t-s),\xi+\la(\xi'-\xi)) \, d\mu \, d\la.$$
Thus, $L\varphi(s,t,\xi,\xi')$ is a bilinear mapping on $\R \times (\R \times \R^d)$ such that $\norm{L\varphi}_\infty \leq \norm{D^2\varphi}_\infty$ and $\norm{L\varphi(s,t,\xi_1,\xi_1')-L\varphi(s,t,\xi_2,\xi_2')} \leq \norm{D^2\varphi}_\ka \lp \norm{\xi_1-\xi_2}^\ka+\norm{\xi_1'-\xi_2'}^\ka \rp$.\\
With this notation, it is readily checked that
$$\si(t,\xi)-\si(s,\xi)-\si(t,\xi')+\si(s,\xi')=L\si(s,t,\xi,\xi')((t-s,0),(0,\xi-\xi'))$$
for any $s,t \in [0,T]$, $\xi,\xi' \in [0,T] \times \R^d$, so that
\begin{eqnarray}
\lefteqn{\norm{ [\si^t-\si^s](\cy_u)-[\si^t-\si^s](\cyti_u)-[\si^t-\si^s](\cy_v)+[\si^t-\si^s](\cyti_v)}}  \label{increment-double}\\
&=&\lVert L\si(s,t,\cy_u,\cyti_u)((t-s,0),(0,\cy_u-\cyti_u)) \nonumber\\
& & \hspace{6cm}-L\si(s,t,\cy_v,\cyti_v)((t-s,0),(0,\cy_v-\cyti_v))\rVert  \nonumber\\
&\leq & \lVert L\si(s,t,\cy_u,\cyti_u)((t-s,0),(0,[\cy_u-\cyti_u]-[\cy_v-\cyti_v]))\rVert  \nonumber\\
& & \hspace{3cm} +\lVert [L\si(s,t,\cy_u,\cyti_u)-L\si(s,t,\cy_v,\cyti_v)]((t-s,0),(0,\cy_v-\cyti_v))\rVert    \nonumber\\
&\leq & \norm{D^2\si}_\infty \lln t-s \rrn \norm{[y_u-\yti_u]-[y_v-\yti_v]}  \nonumber\\
& &\hspace{3cm} +\norm{D^2\si}_\ka \lp 2\lln u-v \rrn^\ka+\norm{y_u-y_v}^\ka+\norm{\yti_u-\yti_v }^\ka \rp \lln t-s \rrn \norm{y_v-\yti_v}  \nonumber\\
&\leq & c_\si \lln t-s \rrn \big\{ \cn[y-\yti;\cac_1^\ka] \lln u-v \rrn^\ga   \nonumber\\
& &\hspace{3cm} +\lp 2\lln u-v\rrn^\ka+\lln u-v \rrn^{\ka\ga} \lcl\cn[y; \cac_1^\ga]^\ka+\cn[\yti; \cac_1^\ga]^\ka \rcl \rp \cn[y-\yti; \cac_1^\ga] \, T^\ga \big\}, \nonumber
\end{eqnarray}
which leads to the result.

\end{proof}

\begin{proof}[Proof of Proposition \ref{stabilite-sigma}]
This is a matter of elementary differential calculus. For the sake of conciseness, denote $\si=\si_i$ and $\varphi_{uv}(r)=\cy_u+r(\cy_v-\cy_u)$. Then
\begin{eqnarray}\label{decompo-sigma-de-y}
\lefteqn{(\der(\si^t(\cy))_{uv} \ = \  \si^t(\cy_v)-\si^t(\cy_u)} \nonumber\\
&=& \int_0^1 dr \, D_2\si(t,\varphi_{uv}(r))(v-u)+\int_0^1 dr \, D_3\si(t,\varphi_{uv}(r))(\der y)_{uv} \nonumber\\
&=& D_3\si(t,\cy_u)(\der y)_{uv}+\int_0^1 dr \, \lc D_3\si(t,\varphi_{uv}(r))-D_3\si(t,\cy_u) \rc(\der y)_{uv} \nonumber\\
& & \hspace{7cm} + \int_0^1 dr \, D_2\si(t,\varphi_{uv}(r))(v-u) \nonumber\\
& :=& (D_3\si(t,\cy_u) \circ y'_u)(\der x)_{uv}+r_{uv},
\end{eqnarray}
where $r$ has to be interpreted as a remainder, whose exact expression is given by:
\begin{multline*}
r_{uv}=D_3\si(t,\cy_u) r_{uv}^y 
+\int_0^1 dr \, \lc D_3\si(t,\varphi_{uv}(r))-D_3\si(t,\cy_u) \rc(\der y)_{uv}\\
+ \int_0^1 dr \, D_2\si(t,\varphi_{uv}(r))(v-u).
\end{multline*}
We will now bound the two terms in expression (\ref{decompo-sigma-de-y}).

\smallskip

First, $\norm{D_3\si(t,\cy) \circ y' }_\infty \leq \norm{D_3\si}_\infty \cn[y'; \cac_1^0]  \leq c_\si \cn[y;\cq^\ga]$, and if $0\leq u <v \leq T$, 
\bean
\lefteqn{\norm{D_3\si(t,\cy_v)\circ y'_v-D_3\si(t,\cy_u)\circ y'_u }}\\
&\leq &  \norm{ \lc D_3\si(t,\cy_v)-D_3\si(t,\cy_u)\rc \circ y'_v }+\norm{D_3\si(t,\cy_u)\circ \lc y'_v-y'_u\rc }\\
& \leq & \norm{D^2\si}_\infty \norm{\cy_v-\cy_u } \cn[y';\cac_1^0]+\norm{D_3\si}_\infty \cn[y';\cac_1^\ga] \lln v-u\rrn^\ga\\
&\leq & \norm{D^2\si}_\infty (\lln v-u \rrn+\cn[y;\cac_1^\ga] \lln v-u\rrn^\ga) \cn[y';\cac_1^0]+\norm{D_3\si}_\infty \cn[y';\cac_1^\ga] \lln v-u\rrn^\ga\\
&\leq & c_\si \lln v-u\rrn^\ga \lcl 1+\cn[y;\cq^\ga]^2 \rcl,
\eean
hence $D_3\si(t,\cy)\circ y' \in \cac_1^\ga$ and $\cn[D_3\si(t,\cy)\circ y';\cac_1^\ga]\leq c_\si \lcl 1+\cn[y;\cq^\ga]^2 \rcl$.

\smallskip

As for $r$, if $0\leq u < v \leq T$,
\bean
\norm{r_{uv}}& \leq & \norm{D_3\si}_\infty \cn[r^y;\cac_2^{2\ga}] \lln v-u\rrn^{2\ga}+\norm{D^2\si}_\infty \norm{\cy_v-\cy_u } \cn[y;\cac_1^\ga] \lln v-u\rrn^\ga\\
& & \hspace{9cm}+\norm{D_2\si}_\infty \lln v-u\rrn\\
& \leq & c_\si \lln v-u\rrn^{2\ga} \lcl 1+\cn[y;\cq^\ga]^2 \rcl,
\eean
so that $r \in \cac_2^{2\ga}$ and $\cn[r; \cac_2^{2\ga}] \leq c_\si \lcl 1+\cn[y;\cq^\ga]^2 \rcl$. 

\smallskip

To get (\ref{eq:stabilite-sigma}), it only remains to notice that $\cn[\si^t(\cy);\cac_1^\ga]\leq c_\si \lcl 1+\cn[y;\cq^\ga] \rcl$.

\end{proof}

\begin{proof}[Proof of Lemma \ref{estimation-sigma-rough}]
According to the proof of Proposition \ref{stabilite-sigma}, if $D_1\si^t:=D_2\si(t,.,.)$ and $D_2\si^t:=D_3\si(t,.,.)$, one has $[\si_i^t-\si_i^s](\cy)'_u=D_2(\si_i^t-\si_i^s)(\cy_u) \circ y'_u$ and 
\begin{multline*}
r_{uv}^{[\si_i^t-\si_i^s](\cy)}=D_2[\si_i^t-\si_i^s](\cy_u)(r^y_{uv})\\+\int_0^1 dr \, [D_2(\si^t-\si^s)(\cy_u+r(\cy_v-\cy_u))-D_2(\si^t-\si^s)(\cy_u)] (\der y)_{uv}
+\int_0^1 dr \, D_1(\si^t-\si^s)(\cy_u)(v-u).
\end{multline*}
Recall that in order to bound $(\si_i^t-\si_i^s)(\cy_u)$ in $\cq^\ga$, the main steps consist in estimating $\cn[(\si_i^t-\si_i^s)(\cy_u)';\cac_1^\ga]$ and $\cn[r;\cac_2^{2\ga}]$. However,
\bean
\lefteqn{\norm{[\si_i^t-\si_i^s](\cy)'_v-[\si_i^t-\si_i^s](\cy)'_u]}}\\
&\leq & \norm{ [D_2(\si_i^t-\si_i^s)(\cy_v)-D_2(\si_i^t-\si_i^s)(\cy_u)] \circ y'_v }+\norm{D_2(\si_i^t-\si_i^s)(\cy_u) \circ [y'_v-y'_u]}\\
&\leq & \norm{D^2(\si_i^t-\si_i^s)}_\infty (\lln v-u \rrn+\cn[y;\cac_1^\ga] \lln v-u \rrn^\ga )\cn[y';\cac_1^0]\\
& & \hspace{7cm} +\norm{D_2(\si_i^t-\si_i^s)}_\infty \cn[y';\cac_1^\ga] \lln v-u\rrn^\ga\\
&\leq & \norm{D^3 \si_i} \lln t-s \rrn (\lln v-u\rrn+\lln v-u \rrn^\ga \cn[y;\cac_1^\ga]) \cn[y';\cac_1^0]\\
& & \hspace{7cm} +\norm{D^2\si_i}_\infty \lln t-s \rrn \cn[y';\cac_1^\ga] \lln v-u\rrn^\ga\\
&\leq & c_\si \lln t-s \rrn \lln v-u \rrn^\ga \lcl 1+\cn[y;\cq^\ga]^2 \rcl,
\eean
and 
\bean
\norm{r_{uv}^{[\si_i^t-\si_i^s](\cy)} } &\leq & \norm{D_1(\si_i^t-\si_i^s)}_\infty \lln v-u \rrn +\norm{D_2(\si_i^t-\si_i^s)}_\infty \lln v-u \rrn^{2\ga} \cn[r^y;\cac_2^{2\ga}]\\
& & \hspace{1.5cm} +\norm{D^2(\si_i^t-\si_i^s)}_\infty (\lln v-u\rrn+\cn[y;\cac_1^\ga] \lln v-u \rrn^\ga) \cn[y;\cac_1^\ga] \lln v-u\rrn^\ga \\
&\leq & c_\si \lln t-s \rrn \lln v-u\rrn^{2\ga} \lcl 1+\cn[y;\cq^\ga]^2 \rcl.
\eean
The upper bound (\ref{estimation-terme-retard-1}) is now easily obtained.

\smallskip

Inequality (\ref{estimation-terme-retard-2}) is in fact a direct consequence of \cite[Proposition 4]{Gu}. Indeed, if $y\in \cq^\ga([0,T];\R^{1,d})$, then of course $\cy \in \cq^\ga([0,T];\R^{1,d+1})$ with decomposition 
$$
(\der \cy)_{st}=(0,y'_s)(\der x)_{st}+(t-s,r_{st}^y).
$$
Then, according to the aforementioned proposition,
$$\cn[\si^t(\cy)-\si^t(\cyti);\cq^\ga]\leq c_{\si,x} \lcl 1+\cn[\cy;\cq^\ga]^2+\cn[\cyti;\cq^\ga]^2\rcl \cn[\cy-\cyti;\cq^\ga].$$
It is then readily checked that $\cn[\cy;\cq^\ga] \leq c \lcl 1+\cn[y;\cq^\ga]\rcl$ and $\cn[\cy-\cyti;\cq^\ga]=\cn[y-\yti;\cq^\ga]$.

\smallskip

Let us now prove Inequality (\ref{estimation-terme-retard-3}). To this end, denote $\ze^{st}:=D_2(\si_i^t-\si_i^s)$ and use the fact that
$[(\si_i^t-\si_i^s)(\cy)]'=\zeta^{st}(\cy)\circ \cy'.$ 
This yields the decomposition $[(\si_i^t-\si_i^s)(\cy)]'-[(\si_i^t-\si_i^s)(\cyti)]')_{uv}=A_{uv}^{st}+B_{uv}^{st}+C_{uv}^{st}+D_{uv}^{st}$, with
$$A_{uv}^{st}=\der(\ze^{st}(\cy))_{uv} \circ [y'_v-\yti'_v] \quad , \quad B_{uv}^{st}=\ze^{st}(\cy_u) \circ \der([y'-\yti'])_{uv},$$
$$C_{uv}^{st}=[\ze^{st}(\cy_v)-\ze^{st}(\cyti_v)] \circ (\der \yti')_{uv} \quad , \quad D_{uv}^{st}=\der([\ze^{st}(\cy)-\ze^{st}(\cyti)])_{uv} \circ \yti'_u.$$
Owing to the regularity of $\si$, we are in position to apply Lemma \ref{estimation-sigma} with $D_3\si_i$, which gives
\bean
\cn[A^{st},\cac_2^{\ka\ga}] &\leq & \cn[D_2(\si_i^t-\si_i^s)(\cy);\cac_1^\ga] T^{\ga(1-\ka)} \cn[y-\yti;\cq^\ga]\\
&\leq & c_\si \lln t-s \rrn \lcl 1+\cn[y;\cq^\ga] \rcl \cn[y-\yti;\cq^\ga],
\eean
and
\bean
\cn[D^{st};\cac_2^{\ka\ga}] &\leq & \cn[D_2(\si_i^t-\si_i^s)(\cy)-D_2(\si_i^t-\si_i^s)(\cyti);\cac_1^{\ka\ga}] \cn[\yti;\cq^\ga]\\
&\leq & c_\si \lln t-s \rrn \lcl 1+\cn[y;\cq^\ga]^\ka+\cn[\yti;\cq^\ga]^\ka \rcl \cn[y-\yti;\cq^\ga] \cn[\yti;\cq^\ga].
\eean

Besides, it is easy to see that $\cn[B^{st};\cac_1^{\ka\ga}] \leq c_\si \lln t-s \rrn \cn[y-\yti;\cq^\ga]$, while $\cn[C^{st};\cac_1^{\ka\ga}] \leq c_\si \lln t-s \rrn \cn[\yti;\cq^\ga] \cn[y-\yti;\cq^\ga]$, hence
\begin{multline}\label{estimation-sigma-derivee}
\cn[([\si_i^t-\si_i^s](\cy)-[\si_i^t-\si_i^s](\cyti))';\cac_1^{\ka\ga}]\\
\leq c_\si \lln t-s \rrn \lcl 1+\cn[y;\cq^\ga]^{1+\ka}+\cn[\yti;\cq^\ga]^{1+\ka} \rcl \cn[y-\yti;\cq^\ga]. 
\end{multline}

As for $r^{st}_{uv}:=r^{[\si_i^t-\si_i^s](\cy)}-r^{[\si_i^t-\si_i^s](\cyti)}$, we know from (\ref{decompo-sigma-de-y}) that, if $\vp_{uv}(r)=\cy_u+r(\cy_v-\cy_u)$, $\tilde{\vp}_{uv}(r):=\cyti_u+r(\cyti_v-\cyti_u)$ and $\si_i^{st}:=\si^t-\si^s$, then $r^{st}_{uv}=r^{st,1}_{uv}+r^{st,2}_{uv}+r^{st,3}_{uv}$, with
$$r^{st,1}_{uv}=\int_0^1 dr \, [D_1\si_i^{st}(\vp_{uv}(r))-D_1\si_i^{st}(\tilde{\vp}_{uv}(r))](v-u),$$
$$r^{st,2}_{uv}=D_2\si_i^{st}(\cy_u)(r^y_{uv})-D_2\si_i^{st}(\cyti_u)(r^{\yti}_{uv}),$$
$$r^{st,3}_{uv}=\int_0^1 dr \, \{ [D_2 \si_i^{st}(\vp_{uv}(r))-D_2\si_i^{st}(\cy_u)](\der y)_{uv}-[D_2 \si_i^{st}(\tilde\vp_{uv}(r))-D_2\si_i^{st}(\cyti_u)](\der \yti)_{uv} \}.$$
Obvious arguments allow to assert that $\cn[r^{st,1};\cac_2^{\ga+\ga\ka}] \leq c_\si \lln t-s \rrn \cn[y-\yti;\cq^\ga]$. To deal with $r^{st,2}$, write of course 
$$r^{st,2}_{uv}=[D_2\si_i^{st}(\cy_u)-D_2\si_i^{st}(\cyti_u)](r^y_{uv})+D_2\si_i^{st}(\cyti_u)([r^y_{uv}-r^{\yti}_{uv}]),$$ 
which leads to $\cn[r^{st,2};\cac_2^{\ga+\ga\ka}] \leq c_\si \lln t-s \rrn \lcl 1+\cn[y;\cq^\ga]\rcl \cn[y-\yti;\cq^\ga]$. Finally, decompose $r^{st,3}$ into $r^{st,3}=r^{st,3,1}+r^{st,3,2}$, with
$$r^{st,3,1}_{uv}=\int_0^1 dr \, [D_2\si_i^{st}(\vp_{uv}(r))-D_2\si_i^{st})(\cy_u)] \der(y-\yti)_{uv},$$
$$r^{st,3,2}_{uv}=\int_0^1 dr \, \lc D_2\si_i^{st}(\vp_{uv}(r))-D_2\si_i^{st}(\cy_u)-D_2\si_i^{st}(\tilde{\vp}_{uv}(r))+D_2\si_i^{st}(\cyti_u)\rc (\der \yti)_{uv}.$$
Clearly, $\cn[r^{st,3,1};\cac_2^{\ga+\ga\ka}] \leq c_\si \lln t-s \rrn \lcl 1+\cn[y;\cq^\ga]\rcl \cn[y-\yti;\cq^\ga]$. To conclude with, observe that the double increment appearing into brackets in $r^{st,3,2}_{uv}$ can be dealt with just as (\ref{increment-double}) (replace $[\si^t-\si^s]$ with $D_2[\si_i^t-\si_i^s]$ and $\cy_v$ with $\vp_{uv}(r)$). This gives 
$$\cn[r^{st,3,2};\cac_2^{\ga+\ga\ka}] \leq c_\si \lln t-s \rrn \lcl 1+\cn[y;\cq^\ga]^\ka+\cn[\yti;\cq^\ga]^\ka \rcl \cn[y-\yti;\cq^\ga] \cn[\yti;\cq^\ga].$$
We have thus shown that
$$\cn[r^{st};\cac_2^{\ga+\ga\ka}] \leq c_\si \lln t-s \rrn \lcl 1+\cn[y;\cq^\ga]^{1+\ka}+\cn[\yti;\cq^\ga]^{1+\ka} \rcl \cn[y-\yti;\cq^\ga],$$
which, together with (\ref{estimation-sigma-derivee}), entails (\ref{estimation-terme-retard-3}).

\end{proof}

\end{document}